\documentclass[11pt]{amsart}

\makeatletter
\usepackage{amssymb}
\usepackage{enumitem}
\usepackage{booktabs}
\usepackage{colortbl}
\usepackage{array}
\usepackage{xcolor}
\usepackage[margin=1.2in]{geometry}
\usepackage[skins,breakable]{tcolorbox}
\usepackage{tikz}
\usetikzlibrary{arrows.meta}
\usepackage[mathcal]{euscript}

\definecolor{deepblue}{RGB}{47,62,85}
\definecolor{midblue}{RGB}{65,99,137}
\definecolor{urlpurple}{RGB}{98,74,131}
\definecolor{palepurple}{RGB}{236,230,246}
\colorlet{keyframe}{urlpurple!65!white}
\colorlet{keyfill}{palepurple!28!white}

\usepackage{hyperref}
\hypersetup{colorlinks=true, linkcolor=deepblue, citecolor=midblue, urlcolor=urlpurple,
  pdftitle={Classification of bipartite observables III: simplicity, mixtures and quantum error correction},
  pdfauthor={Eduardo Due\~nez and Jos\'e Iovino}}

\RequirePackage{amssymb}
\RequirePackage{mathrsfs}

\newcommand{\CC}{\mathbb{C}}

\newcommand{\FF}{\mathbb{F}}

\newcommand{\MM}{\mathbb{M}}
\newcommand{\NN}{\mathbb{N}}

\newcommand{\RR}{\mathbb{R}}

\newcommand{\cP}{\mathcal{P}}

\newcommand{\cR}{\mathcal{R}}

\makeatletter
\RequirePackage{amsmath}
\RequirePackage{amssymb}
\RequirePackage{mathtools}

\DeclareMathOperator{\spn}{span}
\newcommand{\one}{\mathbf{1}}

\makeatletter
\RequirePackage{amsmath}

\providecommand{\tp}{\operatorname{tp}}

\makeatletter
\def\di@Ind#1#2{#1\setbox0=\hbox{$#1x$}\kern\wd0\hbox to 0pt{\hss$#1\mid$\hss}%
  \smash{\lower.9\ht0\hbox to 0pt{\hss$#1\smile$\hss}}\kern\wd0}
\newcommand{\ind}{\mathrel{\mathpalette\di@Ind{}}}
\makeatother

\makeatletter
\RequirePackage[nosort]{cleveref}

\newcommand{\di@crefnames}[3]{\crefname{#1}{#2}{#3}\Crefname{#1}{#2}{#3}}
\di@crefnames{section}{Section}{Sections}
\di@crefnames{subsection}{Section}{Sections}
\di@crefnames{subsubsection}{Section}{Sections}
\di@crefnames{chapter}{Chapter}{Chapters}
\di@crefnames{part}{Part}{Parts}
\di@crefnames{appendix}{Appendix}{Appendices}
\di@crefnames{table}{Table}{Tables}
\di@crefnames{figure}{Figure}{Figures}
\di@crefnames{theorem}{Theorem}{Theorems}
\di@crefnames{keytheorem}{Theorem}{Theorems}
\di@crefnames{proposition}{Proposition}{Propositions}
\di@crefnames{corollary}{Corollary}{Corollaries}
\di@crefnames{lemma}{Lemma}{Lemmas}
\di@crefnames{fact}{Fact}{Facts}
\di@crefnames{conjecture}{Conjecture}{Conjectures}
\di@crefnames{definition}{Definition}{Definitions}
\di@crefnames{example}{Example}{Examples}
\di@crefnames{remark}{Remark}{Remarks}
\di@crefnames{question}{Question}{Questions}
\di@crefnames{axiom}{Axiom}{Axioms}
\di@crefnames{notation}{Notation}{Notations}
\di@crefnames{exercise}{Exercise}{Exercises}

\AtBeginDocument{%
}
\makeatletter
\def\@currname{di-theorem-rudin}\def\@curroptions{}
\RequirePackage{amsthm}
\RequirePackage{aliascnt}

\def\di@rudin@parent{section}
\relax

\newtheoremstyle{rudin}%
  {\topsep}{\topsep}{\itshape}{}{\bfseries}{.}{ }%
  {\thmname{#1}\thmnumber{ #2}\thmnote{ (#3)}}
\newtheoremstyle{rudindef}%
  {\topsep}{\topsep}{\upshape}{}{\bfseries}{.}{ }%
  {\thmname{#1}\thmnumber{ #2}\thmnote{ (#3)}}

\theoremstyle{rudin}
\ifx\di@rudin@parent\@empty
  \newtheorem{theorem}{Theorem}
\else
  \newtheorem{theorem}{Theorem}[\di@rudin@parent]
\fi
\newcommand{\DiSharedTheorem}[2]{%
  \newaliascnt{#1}{theorem}%
  \newtheorem{#1}[#1]{#2}%
  \aliascntresetthe{#1}%
}
\DiSharedTheorem{proposition}{Proposition}
\DiSharedTheorem{corollary}{Corollary}
\DiSharedTheorem{lemma}{Lemma}
\DiSharedTheorem{fact}{Fact}
\theoremstyle{rudindef}
\DiSharedTheorem{definition}{Definition}
\DiSharedTheorem{example}{Example}
\DiSharedTheorem{remark}{Remark}
\DiSharedTheorem{question}{Question}
\makeatletter

\theoremstyle{rudin}
\DiSharedTheorem{keytheorem}{Theorem}
\tcolorboxenvironment{keytheorem}{%
  breakable, colback=keyfill, colframe=keyframe, boxrule=0.4pt, arc=1pt,
  left=6pt, right=6pt, top=2pt, bottom=2pt,
  before skip=10pt, after skip=10pt}

\theoremstyle{rudindef}

\newcommand{\forphysicist}{\par\medskip\noindent{\color{urlpurple!45}\rule{\linewidth}{0.4pt}}\par\smallskip\noindent\textcolor{deepblue}{\textbf{For the physicist or operator theorist.}}}
\newcommand{\forlogician}{\par\medskip\noindent\textcolor{urlpurple}{\textbf{For the logician.}}}
\makeatother

\title[Classification of bipartite observables III]{Classification of bipartite observables III: simplicity, mixtures and quantum error correction}
\author{Eduardo Due\~nez}
\address{Department of Mathematics, The University of Texas at San Antonio, One UTSA
Circle, San Antonio, TX 78249, USA}
\email{Eduardo.Duenez@utsa.edu}

\author{Jos\'e Iovino}
\address{Department of Mathematics, The University of Texas at San Antonio, One UTSA
Circle, San Antonio, TX 78249, USA}
\email{jose.iovino@utsa.edu}
\date{\today}
\thanks{This version: \today. The most recent version of this paper is available at \url{https://utsa-quantum-math.github.io/resources/bipartite-observables-3-simplicity/}; it may supersede the version deposited on arXiv.}
\subjclass[2020]{Primary 03C45, 81P73; Secondary 03C50, 03C66, 05C80, 81P45}
\keywords{Simple theories, tree property of the second kind, forking, Kim--Pillay theorem, Fraïssé limits, Rado graph, Pauli group, stabilizer codes, quantum error correction, Pauli channels, bipartite observables, real-valued logic}

\begin{document}

\begin{abstract}
  This is the third paper in a series that reads the dividing lines of Shelah’s classification theory as properties of bipartite quantum observables.
  Here we study simplicity through two countable structures with simple unstable theories: the Rado graph, and the symplectic space of labels of Pauli strings of finite support, in which addition records the products of the strings and the symplectic form records their commutation.
  In the second structure, forking characterizes the logical errors of stabilizer codes:
  an error undetectable by the checks of a code is logical exactly when its type over the stabilizer labels and a complete set of logical-operator labels forks over the stabilizer labels, and the independence theorem becomes a rule for gluing syndromes.
  A two-valued relation between countable sets of labels defines a diagonal bipartite observable, whose real-valued theory describes Alice’s probability mixtures of labels and their expected responses at Bob’s probes.
  Mixing changes the classification:
  if the relation has the independence property, this theory has the tree property of the second kind.
  For the observable of the symplectic form, actual mixtures of checks witness this property for the response itself;
  for the observable of the Rado graph, every array that witnesses it for a quantifier-free formula in Bob’s probe contains idealized preparations, which lie outside the original structure.
\end{abstract}

\maketitle

\setcounter{tocdepth}{1}
\tableofcontents

\section{Introduction}

In Part~I~\cite{Duenez-Iovino:2026a} we studied when the expectation value of a bipartite observable does not depend on the order in which two independently prepared quantum systems, Alice’s and Bob’s, are idealized, and in Part~II~\cite{Duenez-Iovino:2026b} we studied NIP, the absence of an infinite sequence of Alice’s preparations on which Bob’s probes realize every finite pattern of high and low values at a fixed positive gap.
Here we turn to simplicity.
Among complete first-order theories, the simple ones contain the stable ones, and simplicity and NIP do not imply one another:
a simple theory is NIP exactly when it is stable.
Simplicity excludes a formula and an infinitely branching tree of its instances whose branches are consistent while, for some fixed integer $k\ge2$, no $k$ children of any node hold together.
The standard simple unstable structure is the Rado graph, the countable homogeneous graph whose finite induced subgraphs include a copy of every finite graph;
it is the almost sure isomorphism type of the random graph on countably many vertices, in which each edge is chosen independently with probability one half \cite{Fraisse:1954a,Rado:1964,Erdos-Renyi:1963}.

Quantum error correction supplies a second simple unstable structure.
A Pauli string is a tensor product of identities, bit flips and phase flips on individual qubits, the two-level systems of quantum information.
Up to a phase, a Pauli string is determined by a label in a vector space over the two-element field $\FF_2$;
products of strings correspond to sums of labels, and two strings commute or anticommute according to the value at their labels of a nondegenerate alternating form $\omega$, the symplectic form.
On countably many qubits, the labels of the strings that act nontrivially on only finitely many qubits form a vector space $V$ of countably infinite dimension over $\FF_2$, and the structure $(V,\omega)$, which we call the Pauli group, is the Fraïssé limit of the finite alternating spaces over $\FF_2$ (see \cref{sec:pauli}).
Its theory, like that of the Rado graph, is simple and unstable.
Both facts are known (see the historical remarks), and we include proofs because the independence relations they compute are used below (see \cref{thm:rado-simple,thm:symp-simple}).
In both structures, nonforking independence is the independence given by algebraic closure:
in the Rado graph two sets are independent over a third when they meet only inside it, and in the Pauli group when their spans, together with the third set, meet only in its span.

A stabilizer code stores quantum information in the common eigenvalue-$1$ space of commuting Pauli strings, its checks, and the syndrome of an error is the list of bits that records which checks anticommute with it.
Errors with zero syndrome are undetectable, and they are of two kinds:
stabilizers, which act trivially on the encoded information, and logical errors, which change it.
Forking separates them.
Let $S$ be the set of labels of the stabilizers of a code on finitely many qubits, and let $L$ be a set of logical-operator labels such that $S\cup L$ spans the labels of the undetectable errors;
then an undetectable error $e$ is logical exactly when $\tp(e/S\cup L)$ forks over $S$ (see \cref{prop:logical}).
The base matters:
all logical errors have the same type over $S$, and this type does not fork over the empty set.
Two further translations complete the dictionary of \cref{sec:dictionary}.
Clifford unitaries, which map Pauli strings to Pauli strings under conjugation, implement the homogeneity of the Pauli group, and the existence of Clifford encoding circuits for stabilizer codes follows (see \cref{thm:witt-clifford,cor:encoding}).
The independence theorem becomes the gluing of syndromes:
two syndrome prescriptions on sets of checks whose spans, together with a common set, meet only in the span of the common set are realized by a single error as soon as they agree there (see \cref{prop:glue}).
These are statements about dependence and compatibility, not a rule for decoding.

The second half of the paper passes from relations to observables.
A two-valued relation $\varphi$ between countable sets of labels defines a diagonal observable, which reads $\varphi(a,b)$ when Alice’s register holds the label $a$ and Bob’s holds the label $b$.
Only the diagonal of Alice’s density operator enters, so her preparations are probability mixtures of labels, and the response of a mixture at a probe of Bob is the probability that the relation holds between a label drawn from the mixture and an independent draw from the probe (see \cref{sec:observables}).
For the Pauli observable, given by the symplectic form, Alice’s label is a check, Bob’s is an error, and the probes are the Pauli channels on countably many qubits;
depolarizing noise is among them, and it is concentrated on errors of infinite support (see \cref{prop:channels}).
For the Rado observable, given by the edge relation, the probes are the random sets of vertices.
We study the real-valued theory of the two-sorted structure of Alice’s mixtures and Bob’s probes, with the response as its predicate, in Keisler’s model theory of real-valued structures~\cite{Keisler:2023}.

Mixing changes the classification.
If the relation has the independence property, the theory of mixtures and probes has the tree property of the second kind, TP$_2$:
there is an array of conditions on a probe whose rows are uniformly inconsistent, no two conditions of a row holding together at a fixed gap, while every choice of one condition from each row is consistent (see \cref{thm:collapse}).
The witnessing formula is the difference of two responses.
On a finite set of labels, the difference between the responses of the uniform mixtures on a half of the set and on its complement reaches $1$ exactly when the probe answers $1$ on the half and $0$ on the complement, and two halves that meet in a quarter of the set give differences that cannot both exceed $1/2$;
the independence property makes the choices from distinct rows compatible, and compactness gives an infinite array in an elementary extension.
Thus the simple unstable theories of the label structures do not describe their mixed response structures.
The two observables are then separated by where a witnessing array can be found:
among the actual preparations, the mixtures of the original structure, or only with idealized preparations, the new elements of the sort of preparations in an elementary extension.

In the statement of the main results, labels on $n$ qubits are regarded as elements of the Pauli group through the first $n$ of countably many qubits, and the pure responses of a diagonal observable are the responses of single labels.
A check \emph{fires} at an error when it anticommutes with it.
A diagonal observable has \emph{freely independent} pure responses when every pattern of answers to Alice’s labels is a pure probe;
examples are the Rado observable and the observable of the random bipartite graph, the Fraïssé limit of the finite bipartite graphs with the two sides as two sorts.

\begin{keytheorem}[Main results]\label{thm:main}
  \leavevmode
  \begin{enumerate}[label=\textup{(\Alph*)},leftmargin=*]
  \item
    The theories of the Rado graph and of the Pauli group $(V,\omega)$ are simple and unstable (\cref{thm:rado-simple,thm:symp-simple}).
  \item
    Let $n$ be a positive integer, let $S$ be the set of labels of the stabilizer group of a stabilizer code on $n$ qubits, and let $L$ be a set of logical operators of the code.
    An undetectable error is a logical error exactly when its type over $S\cup L$ forks over $S$;
    all logical errors have the same type over $S$, and this type does not fork over the empty set (\cref{thm:forking-pauli} and \cref{prop:logical}).
    Every isomorphism between finite-dimensional substructures of the Pauli group extends to an automorphism induced by a Clifford unitary on finitely many qubits, and every stabilizer code has an encoding circuit made of Clifford gates
    (\cref{thm:witt-clifford} and \cref{cor:encoding}).
    If two sets of checks are independent over a common set, then any two syndromes, each realized by some error on its own set of checks, that agree on the common set are realized by a single error (\cref{prop:glue}).
    The probes of the Pauli observable are the Pauli channels,
    and the depolarizing noise of every positive rate on countably many qubits is concentrated on the labels of infinite support, which lie outside $V$
    (\cref{prop:channels}).
  \item
    For every diagonal observable whose pure responses have the independence property,
    and in particular for the Rado observable and the Pauli observable, the real-valued theory of the observable has the tree property of the second kind, for the formula that is the difference of two responses.
    Every diagonal observable whose real-valued theory is simple is NIP (\cref{thm:collapse}).
  \item
    For the Pauli observable, the indicator of the syndrome sector of every nonzero syndrome of finitely many linearly independent checks is,
    on pure probes, the difference of the responses of two uniform mixtures of checks.
    The response itself has the tree property of the second kind, witnessed by actual preparations:
    for blocks of checks that are linearly independent all together, the conditions that a given check is the first of its block to fire form the rows of a witnessing array
    (\cref{prop:sector} and \cref{thm:pauli-tp2}).
  \item
    For every diagonal observable with freely independent pure responses, and in particular for the Rado observable, the conditions on a probe given by the instances of a fixed quantifier-free formula with actual mixtures as parameters contain no uniformly $k$-inconsistent sequence, for any integer $k\ge2$.
    Hence, for these observables, every array that witnesses the tree property of the second kind of the formula in~(C) contains idealized preparations (\cref{thm:free,thm:free-qf}, and \cref{cor:rado}).
  \end{enumerate}
\end{keytheorem}

The distinction between~(D) and~(E) comes from the relations among pure responses.
For the Pauli observable, the arbitrary-support error labels form a compact group, and Alice’s pure responses become its characters after an affine change of scale, where a character is a continuous homomorphism into $\{\pm1\}$ (see \cref{prop:channels}).
Products of characters remain characters, so averaging over the finite group generated by independent checks expresses the indicator of a nonzero syndrome sector, on pure probes, as a difference of two uniform-mixture responses (see \cref{prop:sector}).
This identity parallels the stabilizer expression for a syndrome projector \cite{Gottesman:1997}, but requires linear independence of the checks rather than their commutation.
A uniform mixture on the coset of one check modulo the span of its predecessors consequently has response $1$ at a pure probe exactly when that check is the first in its block to fire;
in the bit-flip repetition code, whose checks compare neighboring qubits of a chain, the first check to fire marks the first domain wall of a bit flip, the first boundary between flipped and unflipped qubits (see \cref{exa:repetition}).
For the Rado observable, pure responses are coordinate functions on the whole Cantor cube, and responses to actual mixtures depend on a probe only through its coordinate marginals.
A gliding-hump argument then separates the weights of a sequence of conditions into a common finite part and nearly disjoint finite blocks, and one probe nearly satisfies any prescribed finite number of conditions of a suitable subsequence (see \cref{thm:free,thm:free-qf}).
The Pauli observable thus detects linear compatibility among syndrome bits through actual mixtures,
while the Rado observable forces the parameters of a witnessing array beyond the original preparations.

\paragraph{Organization of the paper.}
\Cref{sec:simple} recalls simple theories and the Kim--Pillay theorem, and \cref{sec:pauli} treats the Rado graph and the Pauli group.
\Cref{sec:dictionary} is the dictionary between the Pauli group and stabilizer codes.
\Cref{sec:observables} introduces diagonal observables, their probe spaces and their real-valued theories, \cref{sec:collapse} proves the tree property of the second kind for mixed preparations, and \cref{sec:pauli-tp2,sec:free} locate its witnesses for the Pauli and the Rado observables.
\Cref{sec:affine} compares mixtures with direct integrals in affine logic, and \cref{sec:questions} collects questions.

\section{Simple theories and the Kim--Pillay theorem}\label{sec:simple}

The independence theorem addresses a compatibility problem:
two nonforking extensions of the same type over a model, to parameter sets independent over that model, admit a common nonforking extension.
For Pauli labels, types specify syndrome bits and linear relations, and this compatibility is expressed by the symplectic form and linear span.
We first recall the model-theoretic formulation, including the tree configurations excluded by simplicity and the Kim--Pillay characterization by an independence relation.
In \cref{sec:pauli} we verify this characterization for the Rado graph and the symplectic space of Pauli labels.

To fix the setting, we work in a monster model $\MM$ of a complete first-order theory $T$:
a model of very large cardinality in which every set of formulas in finitely many variables with parameters from a subset of smaller cardinality is satisfied whenever it is finitely satisfiable,
and in which every elementary map between two such subsets, that is, every map under which each tuple and its image satisfy the same formulas, extends to an automorphism.
These subsets are called small, and from now on all sets of parameters are assumed to be small.
The type $\tp(a/C)$ of a tuple $a$ over a set $C$ is the set of formulas with parameters from $C$ that $a$ satisfies.
For sets $B$ and $D$ we write $BD$ for $B\cup D$, and, identifying a tuple $a$ with the set of its entries, we write $aB$ for the union of the set of its entries with $B$.
For a finite sequence $s\in\NN^{<\NN}$ of natural numbers and
$j\in\NN$, we write $s^\frown j$ for the sequence obtained by appending $j$ to $s$, and for an infinite sequence $\eta\in\NN^\NN$ and $n\in\NN$, we write $\eta\restriction n$ for the sequence of the first $n$ terms of $\eta$.
Two finite sequences are incomparable when neither is an initial segment of the other.

\begin{definition}\label{def:dividing}
  Let $\varphi(x,y)$ be a formula, where $x$ and $y$ are finite tuples of variables.
  \begin{enumerate}[label=\textup{(\arabic*)}]
  \item
    For an integer $k\ge2$, a sequence $(b_j)_{j\in\NN}$ of tuples of $\MM$ of the length of $y$ is \emph{$k$-inconsistent for $\varphi$} if for all $j_1<\dots<j_k$ in $\NN$ there is no $x$ with $\varphi(x,b_{j_1})\wedge\dots \wedge\varphi(x,b_{j_k})$.
  \item
    The formula $\varphi$ has the \emph{tree property} if, for some
    $k\ge2$, there is a tree $(b_s)_{s\in\NN^{<\NN}}$ such that for every $\eta\in\NN^\NN$ the set $\{\varphi(x,b_{\eta\restriction n}):n\in\NN\}$ is consistent and for every $s$ the sequence $(b_{s^\frown j})_{j\in\NN}$ is $k$-inconsistent for $\varphi$.
  \item
    The formula $\varphi$ has the \emph{tree property of the first kind} if there is a tree
    $(b_s)_{s\in\NN^{<\NN}}$ such that for every $\eta\in\NN^\NN$ the set $\{\varphi(x,b_{\eta\restriction n}):n\in\NN\}$ is consistent and for all incomparable $s,t\in\NN^{<\NN}$ the set $\{\varphi(x,b_s),\varphi(x,b_t)\}$ is inconsistent.
  \item
    The formula $\varphi$ has the \emph{tree property of the second kind}
    (TP$_2$) if, for some $k\ge2$, there is an array $(b_{ij})_{i,j\in\NN}$ such that every row $(b_{ij})_{j\in\NN}$ is $k$-inconsistent for $\varphi$ and for every $f\colon\NN\to\NN$ the set $\{\varphi(x,b_{if(i)}):i\in\NN\}$ is consistent.
  \item
    The theory $T$ is \emph{simple} if no formula has the tree property.
  \end{enumerate}
\end{definition}

Recall that the theory $T$ is stable when no formula $\varphi(x,y)$ has the order property, that is, when there are no tuples $a_i$ and $b_i$, for $i\in\NN$, such that, for all $i,j\in\NN$, the formula $\varphi(a_i,b_j)$ holds exactly when $i<j$.
It is NIP when no formula $\varphi(x,y)$ has the independence property, that is, when there are no tuples $b_i$, for $i\in\NN$, such that every subset of $\NN$ has the form $\{i\in\NN:\varphi(a,b_i)\text{ holds}\}$ for some tuple $a$.
Every stable theory is simple, and the tree property of the second kind implies the tree property;
by a theorem of Shelah, a theory is simple exactly when it has neither TP$_2$ nor the tree property of the first kind (see \cite[Theorem~III.7.11]{Shelah:1990a}).
A simple theory is stable if and only if it is NIP (see \cite{Shelah:1980b,Grossberg-Iovino-Lessmann:2002}), so a simple unstable theory has the independence property.
To prove simplicity we do not build trees;
instead, we exhibit an independence relation directly and show that it is nonforking independence, which we now recall.
For a formula $\varphi(x,y)$ and a tuple $b$ of $\MM$
of the length of $y$, the formula $\varphi(x,b)$ \emph{divides} over a set
$C$ when, for some integer $k\ge2$, there is a sequence
$(b_j)_{j\in\NN}$ with $b_0=b$ that is $k$-inconsistent for $\varphi$ and indiscernible over $C$, that is, such that any two of its increasing finite subsequences of the same length have the same type over $C$.
The formula \emph{forks} over $C$ when it implies a finite disjunction of formulas each of which divides over $C$;
a type forks over $C$ when one of its formulas does.
Nonforking independence is the relation under which $A\ind_C B$ holds, for small sets $A$, $B$ and $C$, when, for every finite tuple $a$ from $A$, the type $\tp(a/BC)$ does not fork over $C$.
By the Kim--Pillay theorem (see \cref{fact:kp}), every relation with the properties listed there coincides with nonforking independence.

We write $|T|$ for the number of formulas of the language of $T$.

\begin{fact}[Kim and Pillay]\label{fact:kp}
  Let $\ind$ be a ternary relation on small subsets of $\MM$, written $A\ind_C B$, and assume that it satisfies the following:
  \begin{enumerate}[label=\textup{(\roman*)},leftmargin=*]
  \item
    \emph{invariance} under automorphisms of $\MM$;
  \item
    \emph{symmetry}:
    the relation $A\ind_C B$ holds if and only if $B\ind_C A$ holds;
  \item
    \emph{monotonicity and transitivity}:
    the relation $A\ind_C BD$ holds if and only if both $A\ind_C B$ and $A\ind_{CB}D$ hold;
  \item
    \emph{finite character}:
    the relation $A\ind_C B$ holds if and only if $a\ind_C b$ holds for all finite tuples $a$ from $A$ and $b$ from $B$;
  \item
    \emph{local character}:
    for every finite tuple $a$ and every set $B$ there is $B_0\subseteq B$ with $|B_0|\le|T|$ and $a\ind_{B_0}B$;
  \item
    \emph{extension}:
    for every tuple $a$ and all sets $C\subseteq B$ there is $a'$ with $\tp(a'/C)=\tp(a/C)$ and $a'\ind_C B$;
  \item
    \emph{independence theorem over models}:
    if $M$ is a small elementary substructure of $\MM$, $B\ind_M D$, $a_1\ind_M B$, $a_2\ind_M D$ and $\tp(a_1/M)=\tp(a_2/M)$, then there is $a$ with $\tp(a/MB)=\tp(a_1/MB)$, $\tp(a/MD)=\tp(a_2/MD)$ and $a\ind_M BD$.
  \end{enumerate}
  Then the theory $T$ is simple and the relation $\ind$ coincides with nonforking independence (see \cite[Theorem~4.2]{Kim-Pillay:1997}).
\end{fact}

\section{The two paradigms}\label{sec:pauli}

The Rado graph and the Pauli labels supply two concrete realizations of the independence calculus.
In each case we identify a Fraïssé limit and use quantifier elimination to describe types by graph relations and linear-algebraic data, respectively.
We prove simplicity and instability for both theories and compute their algebraic closures.
Every set is algebraically closed in the Rado graph, whereas in the Pauli group the algebraic closure of a set is its linear span.
In both cases these closures determine nonforking independence.
The theories of their mixed response structures, whose different behavior depends on the relations among the labels, are compared in \cref{sec:collapse,sec:pauli-tp2,sec:free}.

\subsection{The Rado graph}
The class of finite graphs has the amalgamation property, since two graphs that share an induced subgraph can be glued along it without new edges.
Its Fraïssé limit is the Rado graph $(R,E)$, with vertex set $R$ and edge relation $E$, and a countable graph is isomorphic to the Rado graph if and only if it has the \emph{extension property}:
for all finite disjoint sets $R_1$ and $R_0$ of vertices, there are infinitely many vertices outside $R_1\cup R_0$ adjacent to every vertex of $R_1$ and to no vertex of $R_0$ \cite{Fraisse:1954a,Rado:1964}.
The theory of $(R,E)$ is $\aleph_0$-categorical and eliminates quantifiers, so the type of a tuple over a set $D$ is determined by the equalities and edges among its entries and between its entries and the elements of $D$.
By saturation, the monster model $\MM$ of the theory of $(R,E)$ has the extension property for small disjoint sets $R_1$ and $R_0$ of vertices.
In particular, the algebraic closure of a small set $D$ of vertices of $\MM$ is $D$ itself:
for a vertex $v$ outside $D$, apply the extension property to its neighbors in $D$ and the remaining vertices of $D$.
This gives infinitely many vertices outside $D$ with the same edges to $D$ as $v$.
Quantifier elimination gives them the same type over $D$, so $v$ is not algebraic over $D$.

The relation $E(x,y)$ has the independence property, since, by compactness, it suffices to realize every finite pattern:
for every positive integer $n$, all distinct vertices $a_1,\dots,a_n$ of the Rado graph and every set of indices among $1,\dots,n$, the extension property gives a vertex adjacent to $a_i$ exactly when $i$ is in the set.
Hence the theory of the Rado graph is unstable.

\begin{theorem}\label{thm:rado-simple}
  The theory of the Rado graph is simple, and, for small sets $A$, $B$ and $C$ of vertices of the monster model $\MM$ of this theory, nonforking independence $A\ind_C B$ holds if and only if $A\cap B\subseteq C$.
\end{theorem}

\begin{proof}
  We verify the hypotheses of \cref{fact:kp} for the relation $A\ind_C B$ defined by $A\cap B\subseteq C$.
  The first five are set algebra, and extension and the independence theorem follow from the extension property of $\MM$.
  Invariance, symmetry and finite character are immediate.
  Monotonicity and transitivity hold together, since the conditions $A\cap B\subseteq C$ and $A\cap D\subseteq C\cup B$ say jointly that every element of $A\cap(B\cup D)$ lies in $C$.
  For local character, given a finite tuple $a$ and a set $B$, the finite set $B_0=B\cap a$ satisfies $a\cap B\subseteq B_0$, which is the relation $a\ind_{B_0}B$.

  In the two constructions below, we temporarily group the coordinates whose entries lie in the base before the remaining coordinates, and restore the original order at the end.
  For two tuples with the same type over the base, use the same coordinate permutation.
  Repeated entries retain their equality pattern:
  one new vertex is chosen for each distinct entry outside the base.
  For extension, let $a$ and $C\subseteq B$ be given, and write $a=(c,a_0)$, where $c$ lists the entries of $a$ in $C$.
  The type of $a_0$ over $C$ prescribes only the equalities and edges of $a_0$ within itself and to $C$.
  Hence the extension property of $\MM$, applied to the distinct entries in turn, with
  $B\setminus C$ added to the set of vertices to which nonadjacency is prescribed, gives a tuple $a_0'$ disjoint from $B$ with the same type over $C$, and the tuple $a'=(c,a_0')$ satisfies $a'\cap B\subseteq C$.

  For the independence theorem, let $M$, $B$, $D$, $a_1$ and $a_2$ be as in \cref{fact:kp}(vii).
  Since $\tp(a_1/M)=\tp(a_2/M)$, the two tuples have the same entries in $M$ and the same remaining equalities;
  write $a_\ell=(m,a_\ell')$ for $\ell=1,2$, where the tuple $m$ lists these common entries and each $a_\ell'$ is disjoint from $M$.
  The hypotheses $a_1\ind_MB$ and $a_2\ind_MD$ say that $a_1'$ is disjoint from $B$ and $a_2'$ from $D$.
  We look for a tuple $x$, disjoint from $M\cup B\cup D$, whose equalities and edges within itself and edges to
  $M$ are those of $a_1'$ (equivalently, of $a_2'$), whose edges to $B\setminus M$ are those of $a_1'$, and whose edges to $D\setminus M$ are those of $a_2'$.
  These prescriptions concern disjoint sets of vertices, because $B\cap D\subseteq M$;
  the extension property of $\MM$ therefore provides such an $x$, and by quantifier elimination the tuple $a=(m,x)$ is as required.
\end{proof}

\subsection{Pauli strings}
The Hilbert space of a qubit is $\CC^2$, and the bit flip and the phase flip are two of the Pauli matrices, which anticommute:
\[
  X=\begin{pmatrix}0&1\\1&0\end{pmatrix},\qquad
  Z=\begin{pmatrix}1&0\\0&-1\end{pmatrix},\qquad XZ=-ZX.
\]
For a positive integer $n$, a Pauli string on $n$ qubits is, up to a phase, a tensor product
$X^{x_1}Z^{z_1}\otimes\dots\otimes X^{x_n}Z^{z_n}$ with $x_k$ and $z_k$ in the two-element field $\FF_2$ for $k=1,\dots,n$, and it is labeled by the vector $(x,z)\in\FF_2^{2n}$, where $x=(x_1,\dots,x_n)$ and $z=(z_1,\dots,z_n)$;
products of strings correspond to sums of labels.
For a label $v=(x,z)$ we write $P_v$ for the Hermitian string $\mathrm{i}^{N(v)}X^{x_1}Z^{z_1}\otimes\dots\otimes X^{x_n}Z^{z_n}$, where $N(v)$ is the number of indices $k$ with $x_k=z_k=1$;
it satisfies $P_v^2=I$, where $I$ is the identity, and the label $(1,1)$ of one qubit gives the third Pauli matrix $P_{(1,1)}=\mathrm{i}XZ=Y$.
Two strings either commute or anticommute, and
\[
  P_{(x,z)}P_{(x',z')}=(-1)^{\omega((x,z),(x',z'))}P_{(x',z')}P_{(x,z)},
  \qquad \omega((x,z),(x',z'))=x\cdot z'+z\cdot x',
\]
where $x\cdot z'=\sum_{k=1}^n x_kz'_k$, computed in $\FF_2$.
The form $\omega$ is bilinear, alternating and nondegenerate on $\FF_2^{2n}$.
Recall that in error correction an error is a Pauli string that acts on an encoded state as noise, and a check is a Pauli string whose measurement detects errors;
we identify each Pauli string, up to its phase, with its label.
The \emph{syndrome} of an error $e$ with respect to checks $g_1,\dots,g_m$, for a positive integer $m$, is the vector $(\omega(g_1,e),\dots,\omega(g_m,e))$:
when the checks commute pairwise, measuring them on a state that they fix, after the error has acted on it, reveals which of them anticommute with the error.
On countably many qubits we write $V$ for the direct sum $\bigoplus_{k\in\NN}\FF_2^2$, with the form given by the same formula, now a finite sum for each pair of vectors;
its elements are the labels of the Pauli strings that act nontrivially on only finitely many qubits.
The label map identifies the group of finite-support Pauli strings modulo its scalars $\{\pm I,\pm\mathrm{i}I\}$ with the group $(V,+)$, and the form $\omega$ records commutation;
we call the structure $(V,\omega)$ the \emph{Pauli group}.

\subsection{The Pauli group as a Fraïssé limit}
In the language with $+$, $0$ and a binary relation for $\omega(x,y)=1$, substructures are $\FF_2$-subspaces with the restricted form.
Finite alternating spaces over $\FF_2$, the finite vector spaces over $\FF_2$ with an alternating bilinear form, amalgamate:
if $W_0\subseteq W_1$ and $W_0\subseteq W_2$, write $W_1=W_0\oplus W_1'$ and $W_2=W_0\oplus W_2'$, and give $W_0\oplus W_1'\oplus W_2'$ the bilinear form that extends the forms of $W_1$ and $W_2$ and vanishes on $W_1'\times W_2'$ and on $W_2'\times W_1'$.
Every finite alternating space embeds in a nondegenerate one, and by Witt’s theorem \cite{Artin:1957} every isometry between subspaces of a finite nondegenerate alternating space extends to an isometry of the whole space.
A back-and-forth argument along symplectic bases therefore shows that $(V,\omega)$ is the Fraïssé limit of the finite alternating spaces over $\FF_2$.
Its theory is $\aleph_0$-categorical and eliminates quantifiers, so the type of a tuple $a$ over a set $D$ is determined by the linear relations of $a$ over $\spn(D)$ and the values of $\omega$ on $\spn(aD)$.
We show after \cref{lem:generic} that the algebraic closure of a set is its span.

\begin{lemma}\label{lem:generic}
  Let $n$ be a positive integer.
  If $a_1,\dots,a_n\in V$ are linearly independent and $\varepsilon_1,\dots,\varepsilon_n\in\FF_2$, then there is $b\in V$ with $\omega(a_i,b)=\varepsilon_i$ for $i=1,\dots,n$.
\end{lemma}

\begin{proof}
  The vectors $a_i$ lie in the span $W$ of the labels of $X$ and $Z$ on the first $m$ qubits, for a large enough integer $m$, and the form is nondegenerate on $W$.
  Hence the functionals $\omega(a_i,\cdot)$ restricted to $W$ are linearly independent, and the linear map $b\mapsto(\omega(a_i,b))_{i\le n}$ from $W$ to $\FF_2^n$ is onto.
\end{proof}

By \cref{lem:generic} and compactness, the formula $\omega(x,y)=1$ has the independence property:
every syndrome with respect to finitely many linearly independent checks is the syndrome of some error.
Hence the theory of $(V,\omega)$ is unstable.

We also use the finite-extension property of the Fraïssé limit.
Every finite alternating extension of a finite subspace of $V$ embeds into $V$ over the subspace:
embed the extension into $V$, then use homogeneity to match the given copy of the subspace.
For each pair of finite alternating spaces, this property is expressed by a first-order extension sentence, so it holds in the monster model $\MM$ as well.
Consequently, for every small subspace $D$ of $\MM$ and every positive integer $n$, the conditions describing an alternating extension $D\oplus\spn(x_1,\dots,x_n)$ by $n$ independent vectors are finitely satisfiable:
each finite set of conditions concerns only a finite-dimensional subspace of $D$ and a finite alternating extension of it.
Saturation realizes these conditions over $D$.
In particular, every linear functional $f$ on $D$ has the form $\omega(\cdot,b)$ on $D$ for some $b\in\MM$.

The same property computes algebraic closure.
Let $A$ be a small subset of $\MM$, let $D=\spn(A)$, and let $v$ be a vector of $\MM$ outside $D$.
For a positive integer $n$, prescribe an alternating form on $D\oplus\spn(x_1,\dots,x_n)$ by requiring
\[
 \omega(d,x_i)=\omega(d,v)\quad(d\in D,\ 1\le i\le n),\qquad
 \omega(x_i,x_j)=0\quad(1\le i,j\le n).
\]
The finite-extension property and saturation realize this space over $D$ by vectors linearly independent over $D$.
These vectors are distinct, and quantifier elimination gives each of them the type of $v$ over $D$.
Hence the type of $v$ over $A$ has infinitely many realizations, and $v$ is not algebraic over $A$;
since every element of $\spn(A)$ is the value of a term at elements of $A$, the algebraic closure of $A$ is $\spn(A)$.

\begin{theorem}\label{thm:symp-simple}
  The theory of $(V,\omega)$ is simple, and, for small subsets $A$, $B$ and $C$ of the monster model $\MM$ of this theory, nonforking independence $A\ind_C B$ holds if and only if $\spn(AC)\cap\spn(BC)=\spn(C)$.
\end{theorem}

\begin{proof}
  The proof follows that of \cref{thm:rado-simple}, with linear spans in place of the sets themselves:
  the first five axioms are linear algebra, and extension and the independence theorem come from the freedom to prescribe the form on a complement.
  We verify the hypotheses of \cref{fact:kp} for the relation $A\ind_C B$ defined by $\spn(AC)\cap\spn(BC)=\spn(C)$.
  Invariance and symmetry are immediate, and finite character holds because a vector of $\spn(AC)\cap\spn(BC)$ outside $\spn(C)$ is a combination of finitely many elements of $A$, $B$ and $C$.

  For monotonicity and transitivity, let $Q_1=\spn(AC)$, $Q_2=\spn(BC)$, $Q_3=\spn(BCD)$ and $Q_0=\spn(C)$;
  these subspaces satisfy $Q_2\subseteq Q_3$ and $\spn(ABC)=Q_1+Q_2$.
  It remains to be shown that $Q_1\cap Q_3=Q_0$ if and only if $Q_1\cap Q_2=Q_0$ and $(Q_1+Q_2)\cap Q_3=Q_2$.
  Suppose that $Q_1\cap Q_3=Q_0$.
  Since $Q_2\subseteq Q_3$, the intersection $Q_1\cap Q_2$ is contained in $Q_0$.
  If $u_1+u_2\in Q_3$ with $u_1\in Q_1$ and $u_2\in Q_2$, then the vector $u_1$ also lies in $Q_3$, hence in $Q_1\cap Q_3=Q_0\subseteq Q_2$;
  thus the sum lies in $Q_2$.
  Conversely, every $u\in Q_1\cap Q_3$ lies in $(Q_1+Q_2)\cap Q_3=Q_2$, and therefore in $Q_1\cap Q_2=Q_0$.

  Local character follows from finite dimension.
  Given a finite tuple $a$ and a set $B$, the space $\spn(a)\cap\spn(B)$ is finite, so it is contained in $\spn(B_0)$ for some finite $B_0\subseteq B$.
  If $u+w\in\spn(B)$ with $u\in\spn(a)$ and $w\in\spn(B_0)$, then the vector $u$ lies in $\spn(a)\cap\spn(B)\subseteq\spn(B_0)$.
  It follows that $\spn(aB_0)\cap\spn(B)=\spn(B_0)$, which is the relation $a\ind_{B_0}B$.

  For extension and the independence theorem, we first treat finite tuples.
  For a small tuple, each finite subset of the required type and of the span-independence conditions involves only finitely many coordinates;
  the finite-tuple constructions therefore make these conditions finitely satisfiable, and compactness and saturation realize them together.
  In either construction, choose a maximal subtuple linearly independent over the base span.
  Every omitted coordinate is a linear combination of the chosen coordinates and base parameters, and is restored by the same equation after the construction.
  In the independence theorem, equality of the two types over the base ensures that the same coordinate indices can be chosen and that these restoring equations agree.
  For extension, let $a$ and $C\subseteq B$ be given, with $a$ now linearly independent over $\spn(C)$.
  Choose a complement $B'$ of $\spn(C)$ in $\spn(BC)$.
  Take a tuple $x$ of new vectors of the same length as $a$ and prescribe the alternating form on $\spn(BC)\oplus\spn(x)$.
  On $\spn(xC)$ use the values prescribed by $\tp(a/C)$, and require $x$ to be orthogonal to $B'$.
  The consequence of the finite-extension property established before the theorem realizes this space over $\spn(BC)$ by a tuple $a'$.
  This tuple has the type of $a$ over $C$ and is linearly independent over $\spn(BC)$.
  Hence $\spn(a'C)\cap\spn(BC)=\spn(C)$, which is the relation $a'\ind_C B$.

  The independence theorem reduces to a gluing of functionals.
  Let $M$, $B$, $D$, $a_1$ and $a_2$ be as in \cref{fact:kp}(vii).
  As in the proof of extension, we may assume that $a_1$ is linearly independent over $M$;
  then so is $a_2$, and the hypotheses $a_1\ind_MB$ and $a_2\ind_MD$ make $a_1$ linearly independent over $\spn(MB)$ and $a_2$ over $\spn(MD)$.
  Let $n$ be the common length of $a_1$ and $a_2$, and, for $\ell=1,2$ and $k=1,\dots,n$, let $a_{\ell,k}$ be the $k$th entry of $a_\ell$.
  For each $k$, the functionals $\omega(a_{1,k},\cdot)$ on $\spn(MB)$ and $\omega(a_{2,k},\cdot)$ on $\spn(MD)$ agree on $M$, because $\tp(a_1/M)=\tp(a_2/M)$.
  Since $B\ind_MD$, their domains meet in $\spn(MB)\cap\spn(MD)=M$, and the two functionals glue to a linear functional on $\spn(MBD)$.
  Realize over $\spn(MBD)$ a tuple $a$, linearly independent over $\spn(MBD)$, with these functionals and with the values of $\omega$ on $a$ given by $\tp(a_1/M)$.
  By quantifier elimination, the type of $a$ over $MB$ is that of $a_1$, and its type over $MD$ is that of $a_2$;
  the relation $a\ind_MBD$ holds by construction.
\end{proof}

In both structures, then, nonforking independence is the independence given by algebraic closure:
for small sets $A$, $B$ and $C$, the relation $A\ind_C B$ holds exactly when the closures of $AC$ and $BC$ meet in the closure of $C$.
The independence theorem then reduces to the freedom to prescribe the remaining relations, edges or values of the form, on each side separately.

\begin{remark}\label{rem:fields}
  The proof of \cref{thm:symp-simple} applies over every finite field $\FF_q$, for a prime power $q$, in the language with a unary function for multiplication by each scalar and a binary relation for each value of the form.
  For symmetric forms in characteristic two, and for arbitrary bilinear forms, the identification of the limit through symplectic bases and Witt’s theorem is no longer available,
  but the Fraïssé limit still exists, since the finite spaces with such a form have free amalgamation.
  The verification of \cref{fact:kp} goes through with the values $\omega(x,x)$ prescribed by the type and, for forms that are not symmetric, with the functionals $\omega(x,\cdot)$ and $\omega(\cdot,x)$ glued separately.
\end{remark}

\section{Stabilizer codes as a structure: a dictionary}\label{sec:dictionary}

The dictionary rests on the linear algebra of Pauli labels:
multiplication modulo phases is vector addition, commutation is a symplectic pairing, and a syndrome is the restriction of the pairing to the span of the checks.
These identifications connect three model-theoretic constructions with stabilizer codes.
Finite-substructure homogeneity is expressed by Clifford conjugation,
forking of a single label is detected by a change of linear span over a specified base, and the independence theorem has a finite-dimensional counterpart in the compatibility of syndrome functionals.
After recalling stabilizer codes, we state and prove the corresponding results and give two readings, translating model theory for the physicist and stabilizer-code algebra for the logician.
The Laflamme--Miquel--Paz--Zurek five-qubit code~\cite{Laflamme-Miquel-Paz-Zurek:1996}, independently found by Bennett, DiVincenzo, Smolin and Wootters~\cite{Bennett-DiVincenzo-Smolin-Wootters:1996}, is the running example;
it is the smallest code correcting an arbitrary error on one qubit.
The translations are collected in \cref{tab:dictionary} at the end of the section;
those for the response structure of the Pauli observable, with its mixtures and probes, are collected in \cref{tab:response-dictionary} at the end of \cref{sec:pauli-tp2}.

\subsection{Stabilizer codes in brief}
We recall what the logician needs (see \cite{Gottesman:1997}).
For a positive integer $n$, a \emph{stabilizer group} on $n$ qubits is a group $\mathcal{G}$ of pairwise commuting Pauli strings that does not contain $-I$.
Its common eigenspace for the eigenvalue $1$ is the \emph{code space};
if $\mathcal{G}$ has $n-k$ independent generators, for an integer $k$ with $0\le k\le n$, then the code space has dimension $2^k$ and encodes $k$ qubits.
In the labels of \cref{sec:pauli}, the group $\mathcal{G}$ corresponds, up to signs, to an isotropic subspace $S\subseteq\FF_2^{2n}$, that is, a subspace on which the symplectic form $\omega$ vanishes.
An error is a Pauli string $e$;
measuring independent generators $g_1,\dots,g_{n-k}$ of $\mathcal{G}$ on the corrupted state returns the \emph{syndrome} $(\omega(g_1,e),\dots,\omega(g_{n-k},e))$, and the decoder chooses a recovery operation from the syndrome.
Errors that differ by a stabilizer have the same syndrome, and the decoder need not distinguish them, since they have the same action on the code space up to a phase.
An error with zero syndrome is undetectable.
It is harmless if it lies in $\mathcal{G}$ up to a phase, since it then acts on the code space as a scalar, and it is a \emph{logical error} otherwise.
On labels, the errors with zero syndrome form the \emph{centralizer} $S^\perp=\{e\in\FF_2^{2n}:\omega(s,e)=0\text{ for all }s\in S\}$, which contains $S$ and has dimension $2n-(n-k)=n+k$;
a set $L$ of \emph{logical operators} is a set of $2k$ labels with $S^\perp=S\oplus\spn(L)$, and the values $\omega(\ell,e)$ for $\ell\in L$ record which logical operator an error of $S^\perp$ applies to the encoded qubits.
If $k\ge1$, the \emph{distance} of the code is the least number of qubits on which a logical error acts.
A \emph{Clifford unitary} is a unitary operator $U$ that maps Pauli strings to Pauli strings under conjugation, up to a phase;
on labels, conjugation by $U$ is a linear map of $\FF_2^{2n}$ that preserves $\omega$.

\begin{example}\label{exa:five-qubit}
  The $[[5,1,3]]$ code has the stabilizer generators
  \[
    g_1=XZZXI,\quad g_2=IXZZX,\quad g_3=XIXZZ,\quad g_4=ZXIXZ,
  \]
  where each word stands for the tensor product of its letters (thus $XZZXI$ stands for $X\otimes Z\otimes Z\otimes X\otimes I$), the letters being the Pauli matrices $X$, $Z$ and $Y=\mathrm{i}XZ$ of \cref{sec:pauli} and the identity $I$ of $\CC^2$.
  The code encodes one qubit, with logical operators $\bar{X}=XXXXX$ and $\bar{Z}=ZZZZZ$.
  Each of the fifteen errors that act as $X$, $Y$ or $Z$ on a single qubit has its own nonzero syndrome, and there are exactly fifteen nonzero syndromes of four bits;
  for instance, the error $XIIII$ has syndrome $(0,0,0,1)$ and the error $IIZII$ has syndrome $(0,0,1,0)$.
  The code therefore corrects every error on one qubit, and its distance is at least $3$;
  after \cref{prop:logical} we exhibit a logical error that acts on exactly three qubits, which shows that the distance is exactly $3$.
\end{example}

\subsection{Fraïssé limits and Clifford circuits}
Finite Pauli configurations carry both linear relations among their labels and the alternating form that records commutation.
Every finite alternating commutation matrix is realized on linearly independent labels (see \cref{prop:age}), and every isometry between finite-dimensional label subspaces extends to a Clifford conjugation on finitely many qubits (see \cref{thm:witt-clifford});
the isometry preserves addition and the form, not only the commutation graph.
For isotropic subspaces, whose labels commute pairwise, the signs of chosen generators can also be matched, yielding the encoding statement of \cref{cor:encoding}.
Here a Clifford circuit is a product of elementary Clifford unitaries on one or two qubits;
every finite-qubit Clifford unitary admits such a decomposition up to a phase~\cite{Gottesman:1997,Dehaene-deMoor:2003}.
Recall from \cref{sec:pauli} that $P_v$ is the Hermitian Pauli string with label $v$, and, for a positive integer $i$, write $Z_i$ for the phase flip on the $i$th qubit.

\begin{proposition}\label{prop:age}
  Let $m$ be a positive integer.
  If $\Theta$ is a symmetric $m\times m$ matrix over
  $\FF_2$ with zero diagonal, then there are Pauli strings $P_1,\dots,P_m$ with linearly independent labels, on at most $m$ qubits, such that, for all $i,j\in\{1,\dots,m\}$, the strings $P_i$ and $P_j$ anticommute exactly when $\Theta_{ij}=1$.
\end{proposition}

\begin{proof}
  The matrix $\Theta$ defines an alternating form on $\FF_2^m$.
  Let $W_0$ be its radical and choose a complement $W$ of $W_0$;
  the form is nondegenerate on $W$, which therefore has a symplectic basis of $r$ pairs, where $2r=\dim W$.
  Send the $r$ pairs to the labels of $X$ and $Z$ on the first $r$ qubits, and a basis of $W_0$ to the labels of $X$ on $\dim W_0$ further qubits.
  This linear map is injective and preserves the form, and the images of the standard basis vectors of $\FF_2^m$ are the required strings, on $r+\dim W_0\le m$ qubits.
\end{proof}

\begin{theorem}\label{thm:witt-clifford}
  Let $n$ be a positive integer, let $W$ and $W'$ be subspaces of $\FF_2^{2n}$, and let $f\colon W\to W'$ be a linear bijection with $\omega(f(v),f(w))=\omega(v,w)$ for all $v,w\in W$.
  Then there is a Clifford unitary $U$ on $n$ qubits with $UP_vU^*=\pm P_{f(v)}$ for all $v\in W$.
  If the vectors $v_1,\dots,v_m$ form a basis of $W$, then the unitary $U$ can be chosen with $UP_{v_i}U^*=P_{f(v_i)}$ for all $i\le m$.
\end{theorem}

\begin{proof}
  By Witt’s theorem, the map $f$ extends to a linear bijection $\widetilde{f}$ of $\FF_2^{2n}$ that preserves $\omega$.
  Every such $\widetilde{f}$ is induced by conjugation by a Clifford unitary \cite{Calderbank-Rains-Shor-Sloane:1998,Dehaene-deMoor:2003};
  let $U$ be one.
  Since the operator $UP_vU^*$ is a Pauli string up to a phase, Hermitian and with square $I$, the phase is a sign.
  For the second statement, let $\varepsilon_i\in\{\pm1\}$ be the sign with $UP_{v_i}U^*=\varepsilon_iP_{f(v_i)}$.
  The vectors $f(v_i)$ are linearly independent, so by nondegeneracy there is $b\in\FF_2^{2n}$ with $(-1)^{\omega(f(v_i),b)}=\varepsilon_i$ for all $i$.
  Conjugation by $P_b$ multiplies $P_{f(v_i)}$ by $(-1)^{\omega(f(v_i),b)}$, so the Clifford unitary $P_bU$ has the required property.
\end{proof}

\begin{corollary}\label{cor:encoding}
  Let $n$ be a positive integer.
  Any two stabilizer groups on $n$ qubits with the same number of independent generators are conjugate by a Clifford unitary.
  In particular, for every integer $k$ with $0\le k\le n$, every stabilizer code that encodes $k$ qubits into $n$ has an encoding circuit made of Clifford gates:
  a Clifford unitary that maps the code space of the trivial code, with generators $Z_1,\dots,Z_{n-k}$, onto the code space of the given code.
\end{corollary}

\begin{proof}
  Let $\mathcal{G}$ and $\mathcal{G}'$ be the two stabilizer groups, let
  $m$ be their common number of independent generators, and let $g_1,\dots,g_m$ and $g'_1,\dots,g'_m$ be independent generators of $\mathcal{G}$ and of $\mathcal{G}'$, respectively, with labels $v_i$ and $v'_i$ for $i=1,\dots,m$.
  Both families span isotropic subspaces, so the linear map $v_i\mapsto v'_i$ preserves $\omega$, and \cref{thm:witt-clifford} gives a Clifford unitary $U_0$ with $U_0P_{v_i}U_0^*=P_{v'_i}$ for all $i$.
  The square of each generator is a scalar, $I$ or $-I$, that lies in its group, and it is therefore $I$, since a stabilizer group does not contain $-I$;
  hence $g_i=\varepsilon_iP_{v_i}$ and $g'_i=\varepsilon'_iP_{v'_i}$ for signs $\varepsilon_i,\varepsilon'_i\in\{\pm1\}$.
  As in the proof of \cref{thm:witt-clifford}, there is $b\in\FF_2^{2n}$ with $(-1)^{\omega(v'_i,b)}=\varepsilon_i\varepsilon'_i$ for all $i$, and the Clifford unitary $U=P_bU_0$ satisfies $Ug_iU^*=g'_i$ for all $i$, whence $U\mathcal{G}U^*=\mathcal{G}'$.
\end{proof}

On countably many qubits, every finite-dimensional subspace of $V$ lies in the labels of the first $n$ qubits for some $n$.
Applying \cref{thm:witt-clifford} there, and extending by the identity on the remaining qubits, shows that every isometry between finite-dimensional subspaces of
$(V,\omega)$ extends to an automorphism of $(V,\omega)$ implemented by a Clifford circuit on finitely many qubits.

\forphysicist{}
A Fraïssé limit is a countable
structure whose finite configurations are, up to isomorphism, exactly those of a given kind and in which every symmetry between two finite configurations extends to a symmetry of the whole structure;
any two such structures are isomorphic.
By \cref{prop:age}, every finite pattern of commutation relations occurs among Pauli strings, and by \cref{thm:witt-clifford}, every symmetry between finite pieces extends to a global one, implemented by a Clifford circuit.
Thus the label structure $(V,\omega)$ of finite-support Pauli strings modulo phases is the Fraïssé limit of the finite alternating spaces over $\FF_2$.
These spaces record finite commutation patterns together with linear relations among labels.
Clifford circuits on finitely many qubits implement the symmetries that witness the limit’s homogeneity.
The existence of encoding circuits for stabilizer codes, usually proved by an explicit construction, is the shadow of this homogeneity:
by \cref{cor:encoding}, the five-qubit code is the image of the trivial code with generators
$Z_1,Z_2,Z_3,Z_4$, the phase flips on the first four qubits, under a Clifford circuit, because both groups have four independent generators, all commuting, and the two generating sets therefore have the same commutation pattern.

\forlogician{}
The structure $(V,\omega)$ of
\cref{sec:pauli} is the Pauli group modulo phases, and, for each positive integer $n$, the Clifford unitaries on the first $n$ qubits act on the finite piece $(\FF_2^{2n},\omega)$ through its automorphism group, the symplectic group.
Ultrahomogeneity, which in \cref{sec:pauli} came from Witt’s theorem and a back-and-forth argument, is implemented by Clifford circuits.
On labels a stabilizer code becomes an isotropic subspace that determines the code up to the signs of its generators.
The encoding result in \cref{cor:encoding} follows from homogeneity together with a correction of these signs by Pauli conjugation.
Nondegeneracy of $\omega$ on $\FF_2^{2n}$ gives the correcting label within the code’s own qubits, as in the proof of \cref{thm:witt-clifford}.

\subsection{Forking and logical errors}
In a simple theory, an extension of the type of an element forks when,
in the sense recalled in \cref{sec:simple} and explained after \cref{prop:logical}, it confines the element to one of infinitely many alternatives that the original type cannot tell apart.
The next theorem computes forking for the Pauli group, using the description of nonforking independence in \cref{thm:symp-simple}.

\begin{theorem}\label{thm:forking-pauli}
  Let $e$ be an element of a monster model $\MM$ of the theory of $(V,\omega)$, and let $D\subseteq B$ be small subsets of $\MM$.
  Then the type $\tp(e/B)$ forks over $D$ if and only if $e\in\spn(B)\setminus\spn(D)$.
\end{theorem}

\begin{proof}
  By \cref{thm:symp-simple}, the type $\tp(e/B)$ does not fork over $D$ exactly when $\spn(eD)\cap\spn(B)=\spn(D)$.
  If $e\in\spn(D)$, the spans $\spn(eD)$ and $\spn(D)$ coincide, and the equality holds.
  Otherwise the span of $eD$ is $\spn(D)\oplus\{0,e\}$, and the equality fails if and only if some vector $e+d$ with $d\in\spn(D)$ lies in $\spn(B)$, that is, if and only if $e\in\spn(B)$.
\end{proof}

By \cref{thm:forking-pauli}, commutation data never fork:
the type of an element $e$ of $\MM$ over a small set $B$ records the values $\omega(b,e)$ for $b\in B$, and these values, constrained only by linearity and by the type of $e$ over a subset $D$ of $B$, can be prescribed without forking over $D$.
Forking happens only when $e$ becomes a linear combination of elements of $B$;
for Pauli strings, this means that $e$ becomes, up to a phase, a product of strings in $B$.
In a stabilizer code this criterion separates the two kinds of nonzero undetectable error.
To state the consequence for the distance of a code, we call the number of qubits on which an error acts nontrivially its \emph{weight}.

\begin{corollary}\label{cor:distance}
  Let $n$ be a positive integer, and regard $\FF_2^{2n}$ as the subspace of $V$ formed by the labels of the strings on the first $n$ qubits.
  Let $\mathcal{G}$ be a stabilizer group on $n$ qubits with set of labels $S$, and let $e\in\FF_2^{2n}$ be a nonzero label with zero syndrome.
  Then the type $\tp(e/S)$ forks over $\emptyset$ if and only if $e\in S$.
\end{corollary}

\begin{proof}
  Apply \cref{thm:forking-pauli} with $D=\emptyset$ and $B=S$;
  the span of $S$ is $S$ itself, and $\spn(\emptyset)=\{0\}$.
\end{proof}

Thus the type over $S$ of a nonzero label with zero syndrome forks over $\emptyset$ exactly when the label is harmless.
Since a logical error is a nonzero label with zero syndrome outside $S$, the distance of a code that encodes at least one qubit is the least weight of a nonzero label in $\FF_2^{2n}$ with zero syndrome whose type over $S$ does not fork over $\emptyset$.

The stabilizer group alone does not tell logical errors apart, because they all have the same type over
its set of labels, but the logical operators, once known, characterize them by forking (see \cref{prop:logical} below).

\begin{proposition}\label{prop:logical}
  Let $n$ be a positive integer, regard $\FF_2^{2n}$
  as a subspace of $V$ as in \cref{cor:distance}, let $\mathcal{G}$ be a stabilizer group on $n$ qubits with set of labels $S$, let $L$ be a set of logical operators of its code, and let $e\in\FF_2^{2n}$ be a label with zero syndrome.
  Then the following hold:
  \begin{enumerate}[label=\textup{(\alph*)}]
  \item
    the label $e$ is a logical error if and only if $\tp(e/S\cup L)$ forks over $S$;
  \item
    all logical errors have the same type over $S$, and this type does not fork over $\emptyset$.
  \end{enumerate}
\end{proposition}

\begin{proof}
  (a) Since the syndrome of $e$ is zero, the label $e$ lies in $S^\perp=\spn(S\cup L)$.
  By \cref{thm:forking-pauli} with $D=S$ and $B=S\cup L$, the type $\tp(e/S\cup L)$ forks over $S$ if and only if $e\in\spn(S\cup L)\setminus S=S^\perp\setminus S$, which is the set of logical errors.

  (b) Suppose that the label $e$ is a logical error, that is, $e\in S^\perp\setminus S$.
  By quantifier elimination, the type of $e$ over $S$ is determined by the linear relations of $e$ over $S$, of which there are none since $e\notin S$, and by the values of $\omega$ on $\spn(eS)$.
  These values are fixed, through bilinearity, by the values of $\omega$ on $S$, by $\omega(e,e)$, which is $0$ because $\omega$ is alternating, and by the values $\omega(s,e)$ for $s\in S$, which are $0$ because $e$ lies in $S^\perp$.
  Hence all logical errors have the same type over $S$.
  The label $e$ is nonzero, since the zero label lies in $S$;
  as it has zero syndrome and lies outside $S$, its type over $S$ does not fork over $\emptyset$, by \cref{cor:distance}.
\end{proof}

In the five-qubit code, the errors $g_1g_2=XYIYX$ (up to a phase) and $\bar{X}=XXXXX$ both have zero syndrome.
The first lies in the stabilizer group, whose set of labels we denote by $S$, so its type over $S$ forks over $\emptyset$;
it is harmless.
The second does not, so its type does not fork over $\emptyset$;
it flips the encoded qubit, and it acts on five qubits.
By \cref{exa:five-qubit}, an error of the second kind acts on at least three qubits, and the logical error $\bar{X}g_1=IYYIX$, up to a phase, acts on exactly three;
the distance of the code is therefore $3$.
With $L=\{\bar{X},\bar{Z}\}$, the logical errors $\bar{X}$, $\bar{Z}$ and $\bar{X}g_1$ have the same type over $S$;
over $S\cup L$ each of them lies in $\spn(S\cup L)$, whence its type becomes algebraic and forks over $S$, as \cref{prop:logical}(a) predicts.

\forphysicist{}
Forking is the model-theoretic notion of an observation that singles out one of infinitely many alternatives that the data known in advance cannot tell apart and of which, for some fixed integer $k\ge2$, no $k$ can hold together.
To be precise, in a simple theory, by a theorem of Kim \cite{Kim:1998}, a condition
$\varphi(x,b)$ on an unknown element $x$ forks over a set
$D$ of known parameters exactly when, for some $k\ge2$, the parameter
$b$ is the first term of a sequence $(b_j)_{j\in\NN}$ that is indiscernible over $D$ and $k$-inconsistent for $\varphi$, in the sense that no $k$ of the conditions $\varphi(x,b_j)$ can hold together (see \cref{def:dividing}).
A sequence is indiscernible over $D$ if any two of its increasing finite subsequences of the same length satisfy the same formulas with parameters from $D$.
A type forks over $D$ when one of its conditions does;
such a type narrows the possibilities for the element but need not determine it.
For a single Pauli label $e$, \cref{thm:forking-pauli} gives the exact criterion:
if $D\subseteq B$ are parameter sets, then $\tp(e/B)$ forks over $D$ precisely when $e\in\spn(B)\setminus\spn(D)$.
Thus a forking extension fixes the label as a sum of parameters from $B$, a sum that lies outside $\spn(D)$, although a complete commutation pattern need not fix a label.
For the stabilizer label space $S$ of a finite code and a nonzero error label $e$ with zero syndrome, the base matters.
Over $\emptyset$, the type $\tp(e/S)$ forks exactly when $e\in S$, the harmless case;
a logical error lies in $S^\perp\setminus S$ and its type over $S$ does not fork over $\emptyset$ (see \cref{cor:distance}).
For a code encoding at least one qubit, its distance is the least weight of such a logical error on the code’s qubits.
Now add logical labels $L$ with $S^\perp=S\oplus\spn(L)$.
The same logical error has $\tp(e/S\cup L)$ forking over $S$, because it is in the enlarged span but not in $S$ (see \cref{prop:logical}).
Without these added labels, all logical errors have the same type over $S$;
nonforking extensions to $S\cup L$ remain possible outside $\spn(S\cup L)$, for example on qubits beyond the code, as the next reading explains.
These are statements about dependence and available parameter sets, not a rule for inferring a unique error from its measured syndrome.

\forlogician{}
The theory of $(V,\omega)$ has a vector space pregeometry for algebraic closure, and by \cref{thm:forking-pauli} its forking is the dependence of this pregeometry:
the type of a single element over a set forks over a subset exactly when the element is algebraic over the set but not over the subset.
Thus the theory has SU-rank one, as does the theory of the Rado graph (see \cref{thm:rado-simple}), with linear span in place of equality.
In finite dimension, the information carried by a syndrome is measured by counting:
if $n$ and $m$ are positive integers and $s_1,\dots,s_m$
are linearly independent labels in $\FF_2^{2n}$, then for each syndrome $\sigma\in\FF_2^m$ exactly $2^{2n-m}$ labels $e$ satisfy $\omega(s_i,e)=\sigma_i$ for all $i$, since the map $e\mapsto(\omega(s_i,e))_{i\le m}$ is linear and onto.
This count records codimension, which forking ignores.
With $S$ the set of labels of the stabilizer group of a code, the logical errors of the code are the elements of the centralizer of $S$ outside $S$, and by \cref{prop:logical} they share one type over $S$.
For a set $L$ of logical operators, the nonforking extensions of this type to $S\cup L$ are realized in $(V,\omega)$ by labels outside $\spn(S\cup L)$ that commute with $S$, for instance by errors on qubits that the code does not touch;
the logical errors realize the forking extensions of the same type, which become algebraic once the logical operators are added.

\subsection{The independence theorem and the gluing of syndromes}
The Kim--Pillay independence theorem combines two nonforking extensions of one type over a model when their parameter sets are independent over that model.
For Pauli labels, the corresponding compatibility condition is linear:
the spans of the two sets of checks meet only in the span of their common base, and the two syndrome functionals agree there.
The proposition below states this finite-dimensional counterpart over the span of an arbitrary common set, rather than assuming that the base is a model.
It asserts existence of a label with both prescribed syndromes;
it imposes neither a weight bound nor a choice of decoder.
Pairwise commutation is needed to interpret a set of checks as simultaneous stabilizer measurements, but not for the linear compatibility statement itself.

\begin{proposition}\label{prop:glue}
  \emergencystretch=1em
  Let $n$ be a positive integer, let $C$, $B$ and $D$ be subsets of $\FF_2^{2n}$ with $\spn(BC)\cap\spn(DC)=\spn(C)$, and let $\xi_1$ and $\xi_2$ be linear functionals on $\spn(BC)$ and on $\spn(DC)$, respectively, that agree on $\spn(C)$.
  Then there is a label $e\in\FF_2^{2n}$ with $\omega(s,e)=\xi_1(s)$ for $s\in\spn(BC)$ and $\omega(s,e)=\xi_2(s)$ for $s\in\spn(DC)$.
\end{proposition}

\begin{proof}
  Since $\xi_1$ and $\xi_2$ agree on the intersection of their domains, there is a linear functional $\xi$ on $\spn(BC)+\spn(DC)$ extending both.
  Extend $\xi$ to $\FF_2^{2n}$.
  Because $\omega$ is nondegenerate, every linear functional on $\FF_2^{2n}$ has the form $\omega(\cdot,e)$ for a unique $e$.
\end{proof}

In the five-qubit code, take $C=\emptyset$, $B=\{g_1,g_2\}$ and $D=\{g_3,g_4\}$.
The spans of $B$ and $D$ meet only in $0$, so any two bits of syndrome on $g_1,g_2$ and any two on $g_3,g_4$ are realized together by one error;
by \cref{exa:five-qubit}, every nonzero syndrome is realized even by an error on a single qubit, for instance the syndrome $(1,0\,|\,1,1)$ by $YIIII$.
If $D$ is replaced by $\{g_1g_2\}$, the hypothesis fails, and so does the conclusion:
the syndrome bit of every error with respect to $g_1g_2$ is the sum of its bits with respect to $g_1$ and $g_2$.

\forphysicist{}
For finite-support Pauli labels, compatible syndrome data are linear functionals on the spans of the checks.
If these spans intersect only in the common base span, agreement there gives a joint syndrome, as \cref{prop:glue} states.
The five-qubit example also shows why the overlap must include linear relations:
a prescribed bit on $g_1g_2$ must equal the sum of the prescribed bits on $g_1$ and $g_2$.
Together with the other axioms of \cref{fact:kp}, the model-theoretic independence theorem characterizes simplicity~\cite{Kim-Pillay:1997}.

\forlogician{}
The statement of \cref{prop:glue} is the
counterpart, in finite dimension and over the span of an arbitrary set rather than over a model, of the independence theorem verified in the proof of \cref{thm:symp-simple}:
the whole argument is the gluing of two functionals whose domains, the spans of two sets of checks together with a common set, meet only in the span of the common set.
Its physical content is that the syndrome bits of an error with respect to a set of checks are the values of a linear functional on the span of the checks, and the consistency of prescribed syndromes is therefore a matter of linear algebra on that span;
the surjectivity of the syndrome map of linearly independent checks is \cref{lem:generic}.

\Cref{tab:dictionary} collects the translations of this section.
It concerns the Pauli group $(V,\omega)$ of \cref{sec:pauli}, in which Pauli phases are discarded.
For a positive integer $n$ and a code on $n$ qubits, $S\subseteq\FF_2^{2n}$ is the set of labels of its stabilizer group, $S^\perp$ is the centralizer of $S$ in $\FF_2^{2n}$, and $L$ is a set of logical operators, so that $S^\perp=S\oplus\spn(L)$;
in the two rows on codes, $e\in S^\perp\setminus\{0\}$ is a nonzero error label with zero syndrome.
In the general type rows, $D\subseteq B$ are parameter sets in a sufficiently saturated elementary extension of the label structure, and $e$ is one label in that extension.
The gluing row assumes that the two check spans intersect exactly in their common base span.

\begin{table}[ht]
  \centering
  \small
  \renewcommand{\arraystretch}{1.22}
  \setlength{\tabcolsep}{3pt}
  \arrayrulecolor{urlpurple!55!white}
  \begin{tabular}{@{}>{\raggedright\arraybackslash}p{0.27\textwidth}>{\raggedright\arraybackslash}p{0.51\textwidth}>{\raggedright\arraybackslash}p{0.18\textwidth}@{}}
    \toprule
    \rowcolor{palepurple}
    \textbf{Model theory} & \textbf{Label algebra and codes} & \textbf{Result} \\
    \midrule
    $(V,\omega)$ & finite-support Pauli labels, addition and commutation & \cref{sec:pauli} \\
    finite substructures & alternating spaces; commutation and linear relations & \cref{prop:age} \\
    finite-substructure homogeneity & isometries extend by finite Clifford conjugation & \cref{thm:witt-clifford} \\
    isotropic subspaces, with generator signs & stabilizer codes and Clifford encoding & \cref{cor:encoding} \\
    algebraic closure of $B$ & $\spn(B)$: generated strings modulo phases & \cref{sec:pauli} \\
    $\tp(e/B)$ & syndrome values and product relations over $B$ & \cref{thm:forking-pauli} \\
    $\tp(e/B)$ forks over $D$ & $e\in\spn(B)\setminus\spn(D)$ & \cref{thm:forking-pauli} \\
    $\tp(e/S)$ does not fork over $\emptyset$ & $e\notin S$: a logical error & \cref{cor:distance} \\
    $\tp(e/S\cup L)$ forks over $S$ & $e\notin S$: the same logical error, with $L$ added & \cref{prop:logical} \\
    independence theorem: finite counterpart & syndrome functionals agreeing on the base glue & \cref{prop:glue} \\
    \bottomrule
  \end{tabular}
  \vspace{6pt}
  \caption{From Pauli labels to stabilizer codes: model-theoretic notions and their counterparts in the label algebra and codes.}
  \label{tab:dictionary}
\end{table}

\section{Observables, mixtures and dividing}\label{sec:observables}

We now pass from a relation on labels to an observable and its mixed preparations.
This changes the structure whose theory is being studied:
a probability distribution on labels becomes one preparation, and the response records its average against a probe.
Although the Pauli and Rado label theories are both simple and unstable, their observables have different probe spaces:
linear functionals on the Pauli space and arbitrary sets of Rado vertices describe the pure probes.
We construct the two response structures and formulate the tree properties for their real-valued predicates, preparing the comparison in the following sections.

\subsection{Diagonal observables}
Let $A$ and $B$ be countable sets and let $\varphi\colon A\times B\to\{0,1\}$ be a function.
The \emph{diagonal observable} of $\varphi$ is the operator
\[
  C_\varphi=\sum_{a\in A,\,b\in B}\varphi(a,b)\,|a\rangle\langle a|\otimes
  |b\rangle\langle b|\quad\text{on }\ell^2(A)\otimes\ell^2(B).
\]
Here $|a\rangle\langle a|$ and $|b\rangle\langle b|$ are the orthogonal projections onto the lines spanned by the basis vectors of $\ell^2(A)$ and $\ell^2(B)$ indexed by $a$ and $b$, respectively.
In plain words, Alice’s register holds a label $a$, Bob’s a label $b$, and the observable reads $\varphi(a,b)$.
A density operator $\rho$ of Alice, that is, a positive operator of trace one on $\ell^2(A)$ that describes her preparation, enters only through its diagonal, which is a probability measure $\mu$ on $A$, a \emph{mixture} of her labels:
the preparation draws $a$ at random with law $\mu$.
For a fixed mixture $\mu$, the actual response is the function $r_\mu\colon B\to[0,1]$ defined by
\[
  r_\mu(b)=\sum_{a\in A}\mu(a)\varphi(a,b)\qquad(b\in B).
\]
The operator on Bob’s space obtained by averaging $C_\varphi$ over Alice’s preparation, the effective observable of $\rho$ in the terminology of Part~II~\cite{Duenez-Iovino:2026b}, is the diagonal operator of multiplication by $r_\mu$.
The probe algebra of Part~II is the unital C$^*$-algebra generated by these effective observables,
that is, the smallest algebra of operators on Bob’s space that contains them and the identity and is closed under adjoints and under limits in the operator norm;
it is the commutative C$^*$-algebra of bounded functions on $B$ generated by the constants and the indicators $b\mapsto\varphi(a,b)$, $a\in A$.
Bob’s probes are the states of the probe algebra, the positive linear functionals on it with value $1$ at the identity, and they form the probe space, which carries the weak$^*$ topology, the topology of pointwise convergence on the probe algebra.
The probe space is described through the closure $K_\varphi$ of the set of functions $\varphi(\cdot,b)$, $b\in B$, in $\{0,1\}^A$ with the product topology.
The space $K_\varphi$ is compact metrizable because $A$ is countable.
On this space Radon probability measures are exactly Borel probability measures.

\begin{lemma}\label{lem:probe-space}
  Let $A$ and $B$ be countable sets, let $\varphi\colon A\times B\to\{0,1\}$ be a function, and let $K_\varphi$ be the closure of the set of functions $\varphi(\cdot,b)$, $b\in B$, in $\{0,1\}^A$.
  Then the probe algebra of the diagonal observable
  $C_\varphi$ is isomorphic to the algebra $C(K_\varphi)$
  of continuous complex functions on $K_\varphi$, the probe space is the space $\cP(K_\varphi)$ of Radon probability measures on $K_\varphi$ with the weak$^*$ topology, and the completed response of a mixture $\mu$ is the affine continuous function $\mathfrak{r}_\mu\colon\cP(K_\varphi)\to[0,1]$ given by
  \[
    \mathfrak{r}_\mu(\nu)=\int_{K_\varphi}\sum_{a\in A}\mu(a)\lambda(a)\,d\nu(\lambda)
    \qquad(\nu\in\cP(K_\varphi)).
  \]
\end{lemma}

\begin{proof}
  For each $a\in A$, let $q_a$ be the coordinate function $q_a(\lambda)=\lambda(a)$ on $K_\varphi$.
  These functions are real-valued and separate points of $K_\varphi$.
  The unital complex algebra they generate is therefore closed under complex conjugation and uniformly dense in $C(K_\varphi)$ by the Stone--Weierstrass theorem.
  Restriction along $b\mapsto\varphi(\cdot,b)$ preserves the supremum norm, because this map has dense image in $K_\varphi$, and sends $q_a$ to the indicator $b\mapsto\varphi(a,b)$.
  It thus identifies $C(K_\varphi)$ with the probe algebra.

  By the Riesz representation theorem, a state of $C(K_\varphi)$ is integration against a unique Radon probability measure;
  pointwise convergence on $C(K_\varphi)$ is precisely the stated weak$^*$ topology on these measures.
  For a mixture $\mu$, the series $\sum_a\mu(a)q_a$ converges uniformly, since the norm of a tail is bounded by the corresponding probability mass.
  This function corresponds to the effective observable of the preparation with diagonal $\mu$.
  Integrating it against the measure $\nu$ gives the displayed response formula, which is affine and weak$^*$ continuous in $\nu$.
\end{proof}

The pure probes are the point masses $\delta_\lambda$, $\lambda\in K_\varphi$, and we identify each of them with its point $\lambda$;
a pure probe is thus an \emph{idealized label} of Bob, recording which of Alice’s labels it answers with $1$.
The actual labels, the functions $\varphi(\cdot,b)$ for $b\in B$, are dense among them, since $K_\varphi$ is the closure of the set of these functions.
For $\lambda\in K_\varphi$, we write $\mathfrak{r}_\mu(\lambda)$ as shorthand for $\mathfrak{r}_\mu(\delta_\lambda)$, the completed response at the corresponding pure probe.
For every actual label $b\in B$, the completed response recovers the actual response:
\[
  \mathfrak{r}_\mu(\delta_{\varphi(\cdot,b)})=r_\mu(b).
\]
The response $r_\mu$ is defined on $B$, whereas the completed response $\mathfrak{r}_\mu$ is defined on the probe space, which contains the point masses of the actual labels.

\subsection{The structure of an observable}
The theory of the observable concerns its mixtures and probes, rather than its label structure alone.
For a diagonal observable $C_\varphi$, let $\mathcal{M}_\varphi$ be the following two-sorted structure.
Its preparation sort consists of Alice’s probability mixtures on $A$, and its probe sort is $\cP(K_\varphi)$.
Its only predicate is the response $R(\mu,\nu)=\mathfrak{r}_\mu(\nu)$ with values in $[0,1]$.
As in Part~II we use Keisler’s model theory for real-valued structures~\cite{Keisler:2023}, without prescribing a metric in advance.
We call $C_\varphi$ \emph{simple} when the real-valued theory of $\mathcal{M}_\varphi$ is simple in the sense of Ben Yaacov~\cite{BenYaacov:2003b}.
An elementary extension is a larger structure in which every formula with parameters from $\mathcal{M}_\varphi$ retains its value;
its new elements of the preparation sort are called idealized preparations.
For the Pauli observable, the error labels of arbitrary support already occur as pure probes in the original structure, as we show in \cref{prop:channels};
elementary extension is a further enlargement.
Every tree property of the theory of $\mathcal{M}_\varphi$ is witnessed in an $\aleph_1$-saturated elementary extension.
Here $\aleph_1$-saturation means that a countable set of conditions with parameters is realized whenever each finite subset can be satisfied within every positive error.
We say that actual preparations witness the property when the preparation parameters of a witnessing array lie in $\mathcal{M}_\varphi$ itself and the finite parts of its paths are realized in the original probe sort, in the sense of \cref{def:family-tp}.
For the Pauli observable, this does not require a whole infinite path to be realized by an error of finite support.

\subsection{Dividing for families of functions}
The passage from truth values to expectations expresses inconsistency through uniformly separated real values.
As in Part~II, an independent sequence of responses realizes every finite high/low pattern at two fixed thresholds.
For the tree properties, individual conditions are satisfiable, but fixed-size subfamilies within rows or sibling sets are uniformly inconsistent, while finite portions of each path remain jointly satisfiable.
The common positive gap supplies the inequalities required in the real-valued formulation, and these finite configurations persist in elementary extensions of the response structure $\mathcal{M}_\varphi$.
We state the function-family formulation before constructing the arrays of mixtures.

\begin{definition}\label{def:family-tp}
  Let $Y$ be a set and $\cR$ a family of real functions on $Y$.
  \begin{enumerate}[label=\textup{(\arabic*)}]
  \item
    For an integer $k\ge2$, a sequence $(r_j)_{j\in\NN}$ in $\cR$ is \emph{uniformly
      $k$-inconsistent} if there are reals $\alpha<\beta$ such that each $r_j$ takes a value $\ge\beta$ somewhere on $Y$, and for all $j_1<\dots<j_k$ no $y\in Y$ satisfies $r_{j_\ell}(y)>\alpha$ for $\ell=1,\dots,k$.
  \item
    The family $\cR$ has the \emph{tree property of the second kind}
    (TP$_2$) if there are $k\ge2$, reals $\alpha<\beta$
    and an array $(r_{ij})_{i,j\in\NN}$ in $\cR$ such that every row is uniformly $k$-inconsistent with the constants $\alpha$ and $\beta$, and for every $f\colon\NN\to\NN$ and every $n\in\NN$ there is $y\in Y$ with $r_{if(i)}(y)\ge\beta$ for all $i<n$.
  \end{enumerate}
\end{definition}

The tree property of \cref{def:dividing}, its variant of the first kind, and TP$_2$ all require infinite sequences of instances that are $k$-inconsistent, for some integer $k\ge2$:
the children of a node, or a row.
For a family of functions, we define the tree property and the tree property of the first kind from \cref{def:dividing} in the same way as TP$_2$ in \cref{def:family-tp}:
fix real thresholds $\alpha<\beta$ and express branch consistency by simultaneous values $\ge\beta$ for every finite set of functions along the branch.
Express inconsistency of a specified finite set by the absence of a point where all its functions have values $>\alpha$.
In each of the three properties, the children of a node, or the members of a row, then form a uniformly $k$-inconsistent sequence, with $k=2$ for the tree property of the first kind.
A family that, for no integer $k\ge2$, contains a uniformly $k$-inconsistent sequence therefore has none of the three tree properties.
Applied to the family of responses of an observable, condition~(2) of \cref{def:family-tp} says that actual preparations witness the tree property of the second kind of the atomic formula, in the real-valued logic of Part~II;
the gap between the constants $\alpha$ and $\beta$ of \cref{def:family-tp} is what the definitions of dividing in real-valued logic require \cite{BenYaacov:2003b}.

\subsection{The Pauli observable}
The \emph{Pauli observable} is the diagonal observable of the symplectic form
$\omega$ on $V\times V$, where $V$ is the space of labels of the Pauli strings on countably many qubits that act nontrivially on only finitely many of them (see \cref{sec:pauli}):
Alice and Bob each hold a Pauli label, and the observable reports whether the two Pauli strings anticommute.
We call Alice’s label the check and Bob’s label the error;
the observable then reports one syndrome bit.
For a mixture $\mu$, the response at an error $e$ is the probability that a check drawn with law $\mu$ anticommutes with $e$.
By \cref{lem:generic}, which provides an error with any prescribed syndrome with respect to finitely many linearly independent checks, the observable has the independence property already on pure preparations.

\begin{lemma}\label{lem:pauli-probes}
  For the Pauli observable, the space $K_\omega$ is the group $V^*$ of all linear functionals $V\to\FF_2$.
\end{lemma}

\begin{proof}
  Each $\omega(\cdot,e)$ is linear, and the linear functionals form a closed subset of $\{0,1\}^V$.
  Conversely, a basic neighborhood of a linear functional $\lambda$ prescribes the values of $\lambda$ on finitely many vectors, which may be taken linearly independent, and \cref{lem:generic} gives $e$ with $\omega(\cdot,e)$ in the neighborhood.
\end{proof}

The functionals of $V^*$ have a direct physical meaning.
Let $\hat{V}=\prod_{k\in\NN}\FF_2^2$, with the product topology, be the group of labels of Pauli strings of arbitrary support, the support of a string being the set of qubits on which it acts nontrivially;
it contains $V$ as the strings of finite support.
For $a\in V$ and $u\in\hat{V}$, the value $\omega(a,u)$ is defined by the formula of \cref{sec:pauli}, which is a finite sum because $a$ has finite support.
For $p$ with $0<p\le1$, let $\nu_p$ be the \emph{depolarizing noise} of rate $p$:
the product measure on $\hat{V}$ under which the labels on distinct qubits are independent, each equal to $(0,0)$ with probability $1-p$ and to each of $(1,0)$, $(0,1)$ and $(1,1)$ with probability $p/3$.
A Radon probability measure on $\hat{V}$ specifies a \emph{Pauli channel} on the norm closure of the algebra of operators supported on finitely many qubits:
on each finite set of qubits, the channel averages conjugation by the Pauli strings given by the restrictions of the random label to that set.
We identify the channel with this probability law.
An arbitrary-support label thus specifies compatible actions on all finite sets of qubits;
no infinite tensor product unitary on a chosen Hilbert space is required.

\begin{proposition}\label{prop:channels}
  Let $\hat{V}$ be the group of labels of Pauli strings of arbitrary support, and for $0<p\le1$ let $\nu_p$ be the depolarizing noise of rate $p$.
  Then the following hold:
  \begin{enumerate}[label=\textup{(\alph*)}]
  \item
    the map $u\mapsto\omega(\cdot,u)$ is an isomorphism of topological groups from $\hat{V}$ onto $V^*$, which carries $V$ onto the functionals of the actual errors;
  \item
    the probes of the Pauli observable are the Pauli channels on countably many qubits, and the response of a mixture
    $\mu$ at a Pauli channel $\nu$ is the probability that a check drawn with law $\mu$ anticommutes with an independent error drawn with law $\nu$;
  \item
    the depolarizing noise $\nu_p$ is concentrated on errors of infinite support, and for a label $a\in V$ that acts nontrivially on $w$ qubits we have
    \[
      \mathfrak{r}_{\delta_a}(\nu_p)=\tfrac12\bigl(1-(1-\tfrac{4p}3)^w\bigr).
    \]
  \end{enumerate}
\end{proposition}

\begin{proof}
  (a) The map is a homomorphism.
  Write $X_k$ and $Z_k$ for the labels of the strings $X$ and $Z$ on the qubit $k$.
  Then the values $\omega(X_k,u)$ and $\omega(Z_k,u)$ are the two coordinates of
  $u$ at the qubit $k$, so the map is injective, and every
  $\lambda\in V^*$ is the image of the vector $u=(x,z)$ with $z_k=\lambda(X_k)$ and $x_k=\lambda(Z_k)$ for every $k$, because $\lambda$ and $\omega(\cdot,u)$ are linear and agree on the basis $\{X_k,Z_k:k\in\NN\}$ of $V$.
  The value $\omega(a,u)$ depends on finitely many coordinates of $u$, so the map is continuous, and a continuous bijection between compact Hausdorff spaces is a homeomorphism.

  (b) By \cref{lem:probe-space}, \cref{lem:pauli-probes} and~(a), the probes are the Radon probability measures on $\hat{V}$, and $\mathfrak{r}_\mu(\nu)=\int\sum_a\mu(a)\,\omega(a,u)\,d\nu(u)$, which is the probability stated, by Fubini’s theorem.

  (c) Under $\nu_p$, the events “the error acts nontrivially on the qubit $k$” are independent and have probability $p$, so, by the second Borel--Cantelli lemma, almost every error acts nontrivially on infinitely many qubits.
  At each of the $w$ qubits where the string with label $a$ acts nontrivially, its factor anticommutes with two of the three nontrivial Pauli matrices, hence with the factor of the error with probability $2p/3$, independently of the other qubits.
  The strings anticommute exactly when an odd number of these $w$ events occur, and this has probability $\frac12\bigl(1-(1-\frac{4p}3)^w\bigr)$.
\end{proof}

In what follows we identify $V^*$ with $\hat{V}$ through \cref{prop:channels}(a) and call its elements \emph{idealized errors}.
By \cref{prop:channels}(c), the depolarizing noise $\nu_p$ of each rate $p$, which models the noise of a memory on countably many qubits, is a probe of the Pauli observable concentrated on idealized errors that are not actual:
with probability one, an error drawn from $\nu_p$ acts on infinitely many qubits.
Alice’s checks see such an error through finitely many qubits at a time, and the response of a check of weight $w$, one that acts nontrivially on $w$ qubits, tends to $\frac12$, the value of a fair coin, as $w$ grows.
For $a\in V$, let $\chi_a$ be the character of $V^*$ given by $\chi_a(\lambda)=(-1)^{\lambda(a)}$, a continuous homomorphism of $V^*$ into the multiplicative group $\{\pm1\}$.
Alice’s pure responses are the characters of this compact group up to an affine change:
the pure response $\mathfrak{r}_{\delta_a}$ of a label $a$ takes the value $\lambda(a)=(1-\chi_a(\lambda))/2$ at an idealized error $\lambda$.

\subsection{The Rado observable}\label{sec:rado-obs}
The \emph{Rado observable} is the diagonal observable of the edge relation $E$ on $R\times R$, where $(R,E)$ is the Rado graph of \cref{sec:pauli}:
Alice and Bob each hold a site of a random network of couplings, and the observable reports whether the two sites are coupled.
By the extension property of the Rado graph, which provides a vertex adjacent to every vertex of a finite set and to none of another finite set disjoint from it, the observable has the independence property already on pure preparations.
We say that the diagonal observable $C_\varphi$ of a function
$\varphi\colon A\times B\to\{0,1\}$ has \emph{freely independent} pure responses when $K_\varphi=\{0,1\}^A$, that is, when every pattern of answers to all of Alice’s labels is an idealized label of Bob.

\begin{lemma}\label{lem:rado-probes}
  For the Rado observable, the space $K_E$ is the whole Cantor cube $\{0,1\}^R$.
  Hence the Rado observable has freely independent pure responses, while the Pauli observable does not.
\end{lemma}

\begin{proof}
  A basic open subset of $\{0,1\}^R$ prescribes adjacency to the vertices of a finite set $R_1$ and nonadjacency to those of a finite set $R_0$, and the extension property gives a vertex outside $R_1\cup R_0$ with that pattern.
  By \cref{lem:pauli-probes}, the space $K_\omega$ is the group $V^*$, which is a proper subset of $\{0,1\}^V$.
\end{proof}

The pure probes of the Rado observable are arbitrary sets of neighbors, that is, arbitrary subsets of $R$, and its probes are random sets of neighbors.
The response of a mixture $\mu$ at a probe $\nu$ is the probability that a site drawn with law $\mu$ lies in a set of neighbors drawn, independently, with law $\nu$;
it depends on $\nu$ only through the marginals $\nu\{\lambda:\lambda(a)=1\}$, for $a\in R$, the probabilities that the individual sites are neighbors.
The Pauli observable sees more of a probe:
its pure responses are, up to an affine change, the characters of the compact group of idealized errors, and products of characters are again characters.
In \cref{sec:pauli-tp2,sec:free} we show what this difference does to the tree property.

\section{The theory of mixed preparations has the tree property of the second kind}\label{sec:collapse}

Passing from labels to mixtures changes the structure whose theory is being considered.
For a diagonal observable with the independence property, we show that the real-valued theory of preparations and probes has the tree property of the second kind, witnessed by the difference of two responses.
The construction compares the uniform mixture on a half of a finite set of labels with the uniform mixture on its complement.
Their difference is $1$ precisely when the probe answers $1$ on the first half and $0$ on the second;
two halves meeting in a quarter of the set cannot both give a difference greater than $1/2$.
The independence property supplies compatible answers on disjoint blocks, yielding finite arrays of actual preparation parameters of arbitrarily large size.
Compactness then supplies the infinite array in an elementary extension (see \cref{thm:collapse});
whether its parameters can all be actual preparations depends on the observable.

The construction uses the independence property on actual labels, whereas NIP for the observable concerns all mixed responses on the probe space.
Recall the fixed-gap definition from Part~II for a family $\cR$ of real functions on a set $Y$.
A sequence $(r_n)_{n\in\NN}$ in $\cR$ is independent if there are real thresholds $\alpha<\beta$ with the following property.
For all disjoint finite sets $N_0,N_1\subseteq\NN$, some $y\in Y$ satisfies $r_n(y)\le\alpha$ for $n\in N_0$ and $r_n(y)\ge\beta$ for $n\in N_1$.
The family has the independence property when it contains such a sequence, and is NIP otherwise.
We first separate the analytic fact that identifies the two conditions.
For a diagonal observable $C_\varphi$ of a function $\varphi\colon A\times B\to\{0,1\}$, recall from \cref{lem:probe-space} that $K_\varphi$ is the closure of $\{\varphi(\cdot,b):b\in B\}$ in $\{0,1\}^A$ and that the probe space is $\cP(K_\varphi)$.
For $a\in A$, let $f_a(\lambda)=\lambda(a)$ for $\lambda\in K_\varphi$, and, for a probability distribution $\mu$ on $A$, let $g_\mu=\sum_{a\in A}\mu(a)f_a$;
by \cref{lem:probe-space}, the response of $\mu$ is $\mathfrak{r}_\mu(\nu)=\int g_\mu\,d\nu$ for $\nu\in\cP(K_\varphi)$.

\begin{lemma}[NIP for pure and mixed responses]\label{lem:diagonal-nip-transfer}
  Let $A$ and $B$ be countable sets and let $C_\varphi$ be the diagonal observable of a function $\varphi\colon A\times B\to\{0,1\}$.
  Then the following are equivalent:
  \begin{enumerate}[label=\textup{(\roman*)}]
  \item
    the family $\{\varphi(a,\cdot):a\in A\}$ is NIP on $B$;
  \item
    the family $\{f_a:a\in A\}$ of pure responses is NIP on $K_\varphi$;
  \item
    the family $\{\mathfrak{r}_\mu:\mu\text{ is a probability distribution on }A\}$ of mixed responses is NIP on the probe space $\cP(K_\varphi)$.
  \end{enumerate}
\end{lemma}

\begin{proof}
  The functions $f_a$ take only the values $0$ and $1$, so independence for them is equivalent to realizing every finite pattern of these two values.
  Each such pattern specifies a clopen subset of $K_\varphi$.
  If it is nonempty, it meets the dense set of actual labels $\{\varphi(\cdot,b):b\in B\}$.
  Conversely, every pattern realized by an actual label is realized in $K_\varphi$.
  This proves (i)$\Leftrightarrow$(ii).
  If the pure responses are independent, the mixed responses associated with the point masses on the same labels are independent, by evaluation at the pure probes $\delta_\lambda$.
  Thus (iii) implies (ii).

  Assume (ii).
  Rosenthal’s combinatorial dichotomy for uniformly bounded sequences of real functions~\cite{Rosenthal:1974} gives every sequence of pure responses a subsequence that converges pointwise on $K_\varphi$, since an independent subsequence is excluded.
  This subsequence is weakly Cauchy in the real Banach space $C(K_\varphi)$ of continuous real-valued functions with the supremum norm:
  its values under each continuous linear functional form a Cauchy sequence.
  Indeed, every such functional is integration against a finite signed Radon measure, and the functions are bounded by $1$, so dominated convergence applies.
  Hence $\{f_a:a\in A\}$ is weakly precompact, meaning that it is bounded and every sequence in it has a weakly Cauchy subsequence.

  As in Part~II~\cite{Duenez-Iovino:2026b}, we use Stegall’s theorem: the absolutely convex hull of a weakly precompact set is weakly precompact (see \cite{Rodriguez:2021}).
  Here the absolutely convex hull consists of the finite sums $\sum_j c_j f_{a_j}$ with $\sum_j|c_j|\le1$.
  Its norm closure is also weakly precompact:
  given a sequence $(h_n)_{n\in\NN}$ in the closure, choose functions $q_n$ in the hull with $\|h_n-q_n\|_\infty<2^{-n}$ for $n\in\NN$ and a weakly Cauchy subsequence of $(q_n)$;
  the corresponding subsequence of $(h_n)$ is weakly Cauchy as well.
  Each $g_\mu$ belongs to this closure, because its finite partial sums lie in the hull and converge uniformly, with the omitted probability mass bounding the error.

  Every sequence $(g_{\mu_n})$ therefore has a weakly Cauchy subsequence.
  Since integration against each probe $\nu\in\cP(K_\varphi)$ is a continuous linear functional on $C(K_\varphi)$, the corresponding subsequence of responses $(\mathfrak{r}_{\mu_n})$ converges pointwise on the probe space.
  An independent sequence of these continuous responses cannot have a pointwise convergent subsequence.
  To see this, prescribe alternating low and high values along any subsequence at its fixed thresholds.
  Every finite prescription is realized, and the associated closed subsets of the compact probe space have the finite intersection property;
  hence one probe realizes the entire alternating prescription, preventing convergence there.
  Thus the mixed family is NIP, proving (iii).
\end{proof}

We now use an independent sequence of actual labels to construct the arrays of mixtures.

\begin{keytheorem}[Mixed preparations and the tree property of the second kind]\label{thm:collapse}
  Let $C_\varphi$ be the diagonal observable of a function $\varphi\colon A\times B\to\{0,1\}$ on countable sets $A$ and $B$, and assume that there are distinct labels
  $a_0,a_1,\ldots$ in $A$ such that for all finite disjoint $N_1,N_0\subseteq\NN$ there is a label $b\in B$ with $\varphi(a_n,b)=1$ for $n\in N_1$ and $\varphi(a_n,b)=0$ for $n\in N_0$.
  Then the affine formula
  \[
    \psi(x;y_1,y_2)=R(y_1,x)-R(y_2,x),
  \]
  in which the predicate $R$ is the response, the variable
  $x$ is of the sort of probes and the variables $y_1,y_2$
  are of the sort of preparations, has the tree property of the second kind in the real-valued theory of the structure $\mathcal{M}_\varphi$ of Alice’s mixtures and Bob’s probes:
  in an $\aleph_1$-saturated elementary extension of $\mathcal{M}_\varphi$, the family of instances of $\psi$, as functions on the probes, has the tree property of the second kind in the sense of \cref{def:family-tp}, with $k=2$, $\alpha=1/2$ and $\beta=1$.
  Moreover, every simple diagonal observable is NIP (see \cref{lem:diagonal-nip-transfer}).
\end{keytheorem}

\begin{proof}
  We first build finite arrays in $\mathcal{M}_\varphi$ itself and then pass to an elementary extension by compactness.

  Let $N\ge1$ be an integer, and choose an integer $t$ with $2^t>N$.
  Choose $2^tN$ of the labels $a_n$ and partition them into $N$ blocks $J_0,\dots,J_{N-1}$ of size $2^t$, and identify each block with $\FF_2^t$.
  Choose distinct nonzero vectors $w_0,\dots,w_{N-1}$ in $\FF_2^t$, and for $i,j<N$ let $H_{ij}\subseteq J_i$ be the set of $v\in\FF_2^t$ with $w_j\cdot v=1$, where the dot product $w_j\cdot v$ is computed in $\FF_2$.
  Each $H_{ij}$ has $2^{t-1}$ elements, and for $j\ne j'$ the set $H_{ij}\cap H_{ij'}$ has $2^{t-2}$ elements, since $w_j$ and $w_{j'}$ are linearly independent.
  Let $\mu^+_{ij}$ and $\mu^-_{ij}$ be the uniform probability measures on $H_{ij}$ and on $J_i\setminus H_{ij}$.

  For a probe $\nu$, write $\pi_a(\nu)=\nu\{\lambda:\lambda(a)=1\}\in[0,1]$ for its marginal at the label $a$;
  by \cref{lem:probe-space}, the response is then $R(\mu,\nu)=\sum_a\mu(a)\pi_a(\nu)$.
  With $m=2^t$, we have
  \[
    \psi(\nu;\mu^+_{ij},\mu^-_{ij})=\frac2m\Bigl(\sum_{a\in H_{ij}}\pi_a(\nu)
    -\sum_{a\in J_i\setminus H_{ij}}\pi_a(\nu)\Bigr)\le1 .
  \]
  For $j\ne j'$ in the same row, adding the two values gives
  \[
    \frac4m\Bigl(\sum_{a\in H_{ij}\cap H_{ij'}}\pi_a(\nu)-\sum_{a\in J_i\setminus(H_{ij}\cup
      H_{ij'})}\pi_a(\nu)\Bigr)\le\frac4m\cdot\frac m4=1 ,
  \]
  so the two values are never both greater than $1/2$.
  For every function $f$ from $\{0,\dots,N-1\}$ to itself, the hypothesis applied to the finite disjoint sets
  $\bigcup_iH_{if(i)}$ and $\bigcup_i(J_i\setminus H_{if(i)})$ gives a label $b$, and at the pure probe of $\varphi(\cdot,b)$ every value $\psi(\cdot;\mu^+_{if(i)},\mu^-_{if(i)})$ with $i<N$ equals $1$.

  Thus, for every $N$, the structure $\mathcal{M}_\varphi$ satisfies the following conditions on $2N^2$ preparations $y^\pm_{ij}$, for $i,j<N$:
  for every function $f$ from $\{0,\dots,N-1\}$ to itself, the value of $\sup_x\min_{i<N}\psi(x;y^+_{if(i)},y^-_{if(i)})$ is $1$;
  and for all $i<N$ and all distinct $j,j'<N$, the value of $\sup_x\min\bigl(\psi(x;y^+_{ij},y^-_{ij}),\psi(x;y^+_{ij'},y^-_{ij'})\bigr)$ is at most $1/2$.
  Impose these closed conditions on an infinite array of preparation variables.
  Every finite subset of the conditions concerns a finite subarray and is satisfied by one of the arrays just constructed, so compactness yields an array $(y^\pm_{ij})_{i,j\in\NN}$ in an elementary extension of $\mathcal{M}_\varphi$.
  Pass to an $\aleph_1$-saturated elementary extension containing this array.
  For each $f\colon\NN\to\NN$, the countable set of conditions $\psi(x;y^+_{if(i)},y^-_{if(i)})=1$, for $i\in\NN$, is approximately satisfiable on every finite subset, by the supremum conditions;
  saturation therefore realizes all of them at one probe.
  Together with the row bound, this gives the tree property of the second kind, which is incompatible with simplicity in real-valued logic \cite{BenYaacov:2003b}.

  For the final assertion, suppose that the diagonal observable is not NIP.
  By \cref{lem:diagonal-nip-transfer}, its pure responses on the actual labels $B$ have an independent sequence.
  Since these responses take only the values $0$ and $1$, the sequence satisfies the hypothesis of the theorem, and the construction above gives TP$_2$ in the theory of $\mathcal{M}_\varphi$.
  The theory is therefore not simple, proving the assertion by contraposition.
\end{proof}

The array that witnesses the tree property of the second kind in the proof of \cref{thm:collapse} lies in an $\aleph_1$-saturated elementary extension of $\mathcal{M}_\varphi$.
For the Pauli observable, actual preparations witness the tree property of the second kind already for the response itself (see \cref{sec:pauli-tp2});
for the Rado observable, every array that witnesses it for a quantifier-free formula in a probe variable contains idealized preparations (see \cref{sec:free}).

\section{The Pauli observable: syndrome sectors and the first check to fire}\label{sec:pauli-tp2}

For the Pauli observable, the linear relations among check labels allow an infinite array of actual mixtures to witness the tree property of the second kind for the response predicate itself.
The construction uses the indicator of a syndrome sector, the set of idealized errors with specified commutation values against finitely many checks.
Within each block of linearly independent checks, Alice draws a label uniformly from the coset of a chosen check modulo the span of its predecessors.
The response is $1$ when that check is the first to anticommute with Bob’s error, and $1/2$ when an earlier check already anticommutes with it.
The resulting row conditions are incompatible at a fixed gap even for mixed probes, whereas choices from finitely many rows are realized by a finite-support error when all the checks together are linearly independent.
Thus the array parameters are actual preparations;
an entire infinite path is realized by a pure probe already in the original probe space, whose error label may have infinite support, as specified in \cref{thm:pauli-tp2}(a).

\subsection{Syndrome sectors}
Let $d$ be a positive integer, let $g_1,\dots,g_d$ be linearly independent checks in the space $V$ of labels of Pauli strings of finite support, with span $G$, and let a syndrome $\sigma\in\FF_2^d$ be given.
We regard $\sigma$ as the linear functional on $G$ with $\sigma(g_i)=\sigma_i$, and we call the set $\Sigma_\sigma=\{\lambda\in V^*:\lambda(g_i)=\sigma_i\text{ for }i\le d\}$ of idealized errors with syndrome $\sigma$ its \emph{syndrome sector}.
For $s\in G$ we write $\mathfrak{r}_s$ for the pure response $\mathfrak{r}_{\delta_s}$, the response of the point mass at $s$, and we recall that $\chi_s$ is the character of $V^*$ given by $\chi_s(\lambda)=(-1)^{\lambda(s)}$.

\begin{proposition}\label{prop:sector}
  Let $d$ be a positive integer, let $g_1,\dots,g_d\in V$ be linearly independent checks with span $G$, and let $\sigma\in\FF_2^d$ be a syndrome, regarded as the linear functional on $G$ with $\sigma(g_i)=\sigma_i$ for $i\le d$.
  Then the indicator of the syndrome sector $\Sigma_\sigma$ satisfies, at every pure probe $\lambda\in V^*$,
  \[
    \one_{\Sigma_\sigma}(\lambda)=2^{-d}\sum_{s\in G}(-1)^{\sigma(s)}\chi_s(\lambda)
      =2^{-d}\sum_{s\in G}(-1)^{\sigma(s)}-2^{1-d}\sum_{s\in G}(-1)^{\sigma(s)}\mathfrak{r}_s(\lambda).
  \]
  Consequently, if $\sigma\ne0$, then $\one_{\Sigma_\sigma}=\mathfrak{r}_{\mu^1_\sigma}-\mathfrak{r}_{\mu^0_\sigma}$ on pure probes, where $\mu^1_\sigma$ and $\mu^0_\sigma$ are the uniform probability measures on $\{s\in G:\sigma(s)=1\}$ and on $\{s\in G:\sigma(s)=0\}$;
  and $\one_{\Sigma_0}=1-2\mathfrak{r}_{\mu_G}$, where $\mu_G$ is the uniform probability measure on $G$.
  In particular, for every probe $\nu$, the probability that an error drawn from the Pauli channel $\nu$ has syndrome $\sigma$ is $\mathfrak{r}_{\mu^1_\sigma}(\nu)-\mathfrak{r}_{\mu^0_\sigma}(\nu)$ if $\sigma\ne0$, and $1-2\mathfrak{r}_{\mu_G}(\nu)$ if $\sigma=0$.
\end{proposition}

\begin{proof}
  The map $s\mapsto(-1)^{\sigma(s)}\chi_s(\lambda)=(-1)^{(\sigma+\lambda)(s)}$ is a character of the finite group $G$;
  its sum over $G$ is $2^d$ if the linear functional $\sigma+\lambda$ vanishes on $G$, that is, if $\lambda\in\Sigma_\sigma$, and $0$ otherwise, since a nonzero functional on $G$ takes the value $1$ on exactly half of $G$.
  This is the first equality, and the second follows from $\chi_s=1-2\mathfrak{r}_s$.
  If $\sigma\ne0$, the first sum on the right vanishes, and the two level sets of $\sigma$ have $2^{d-1}$ elements each;
  grouping the second sum by the value of $\sigma$ gives $\mathfrak{r}_{\mu^1_\sigma}-\mathfrak{r}_{\mu^0_\sigma}$.
  If $\sigma=0$, the first sum is $2^d$, and the second is $2^d\mathfrak{r}_{\mu_G}$.
  Both sides of each identity are affine in the probe, and integration against $\nu$, together with \cref{prop:channels}(b), yields the last statement.
\end{proof}

When the checks commute pairwise and carry signs, the operator
$2^{-d}\sum_{s\in G}(-1)^{\sigma(s)}g_s$, where $g_s$ is the element with label $s$ of the group generated by the signed checks, is the projector onto the states of syndrome
$\sigma$ \cite{Gottesman:1997}, and the probability of the outcome $\sigma$ when the checks are measured on a code state corrupted by a Pauli channel $\nu$ is given by \cref{prop:sector}.
The proposition requires no commutation:
the checks need only be linearly independent.
The coefficients of the indicator, as an affine combination of Alice’s responses, have total absolute value at most $1+2=3$, whatever the number $d$ of checks.

\subsection{The first check to fire}
The next theorem arranges the syndrome sectors of the initial segments of blocks of checks in an array that witnesses the tree property of the second kind.

\begin{keytheorem}[The first check to fire]\label{thm:pauli-tp2}
  Let $(g_{im})_{i,m\in\NN}$ be linearly independent vectors of $V$, let $W_{ij}=\spn(g_{im}:m<j)$ for $i,j\in\NN$, and let $\mu_{ij}$ be the uniform probability measure on the coset $g_{ij}+W_{ij}$.
  For the Pauli observable, the following hold:
  \begin{enumerate}[label=\textup{(\alph*)}]
  \item
    for every $f\colon\NN\to\NN$ and every $n\in\NN$, there is an error $e\in V$ with $r_{\mu_{if(i)}}(e)=1$ for all $i<n$, and there is an idealized error $\lambda\in V^*$ with $\mathfrak{r}_{\mu_{if(i)}}(\lambda)=1$ for all $i$;
  \item
    for every $i$, every $j\ne j'$ and every probe $\nu\in\cP(V^*)$, we have $\min\bigl(\mathfrak{r}_{\mu_{ij}}(\nu),\mathfrak{r}_{\mu_{ij'}}(\nu)\bigr)\le 2/3$, and the bound $2/3$ is attained.
  \end{enumerate}
  Consequently, the family of responses of Alice has the tree property of the second kind in the sense of \cref{def:family-tp}, with $k=2$, $\alpha=2/3$ and $\beta=1$.
\end{keytheorem}

\begin{proof}
  We first compute the response of $\mu_{ij}$ at a pure probe $\lambda\in V^*$.
  If $\lambda$ does not vanish on $W_{ij}$, then the functional $\lambda$
  takes the value $1$ on exactly half of the coset $g_{ij}+W_{ij}$, since a nonzero functional on a finite $\FF_2$-space is balanced, and translation by $g_{ij}$ either preserves or exchanges its level sets.
  If $\lambda$ vanishes on $W_{ij}$, then the functional $\lambda$ is constant on the coset, with value $\lambda(g_{ij})$.
  Hence
  \begin{equation}\label{eq:first-check}
  \mathfrak{r}_{\mu_{ij}}(\lambda)=
  \begin{cases}
  1   & \text{if }\lambda(g_{im})=0\text{ for }m<j\text{ and }\lambda(g_{ij})=1, \\
  1/2 & \text{if }\lambda(g_{im})=1\text{ for some }m<j,                         \\
  0   & \text{otherwise.}
  \end{cases}
  \end{equation}
  In words, the response is $1$ exactly when $g_{ij}$ is the first check of the block $(g_{im})_{m\in\NN}$, which we call block $i$, to fire, that is, to take the value $1$ under $\lambda$.

  (a) Let $f$ and $n$ be given.
  The conditions $\lambda(g_{im})=0$ for $m<f(i)$ and
  $\lambda(g_{if(i)})=1$, for $i<n$, concern finitely many linearly independent vectors, so \cref{lem:generic} gives an error
  $e$ with $\omega(\cdot,e)$ satisfying them, and by~\eqref{eq:first-check} the response $r_{\mu_{if(i)}}(e)$ equals $1$ for all $i<n$.
  The conditions for all $i$ together define a nonempty closed subset of $V^*$, since the vectors $g_{im}$ are linearly independent;
  any $\lambda$ in it works.

  (b) By symmetry, we may assume that $j<j'$.
  By~\eqref{eq:first-check}, the pair $\bigl(\mathfrak{r}_{\mu_{ij}}(\lambda),\mathfrak{r}_{\mu_{ij'}}(\lambda)\bigr)$ at a pure probe takes one of the values
  \[
    (1,\tfrac12),\quad(\tfrac12,\tfrac12),\quad(0,\tfrac12),\quad(0,1),\quad(0,0).
  \]
  Indeed, if the first check of block $i$ to fire among $g_{i0},\dots,g_{ij'}$ has index $m$, the pair is $(\tfrac12,\tfrac12)$ for $m<j$, $(1,\tfrac12)$ for $m=j$, $(0,\tfrac12)$ for $j<m<j'$ and $(0,1)$ for $m=j'$;
  if none fires, the pair is $(0,0)$.
  Responses are affine in the probe, so at a mixed probe $\nu$ the pair lies in the convex hull of these five points.
  Every listed pair $(s,t)$ satisfies $s+2t\le2$, and the same inequality holds throughout their convex hull.
  Hence $3\min(s,t)\le s+2t\le2$, so the minimum of the two coordinates is at most $2/3$.
  Equality requires $s=t=2/3$, and this point is $\tfrac23(1,\tfrac12)+\tfrac13(0,1)=(\tfrac23,\tfrac23)$.
  This point is the pair of responses at the probe $\tfrac23\delta_{\lambda_1}+\tfrac13\delta_{\lambda_2}$, where $\lambda_1$ and $\lambda_2$ are idealized errors under which $g_{ij}$ and $g_{ij'}$, respectively, are the first checks of block $i$ to fire;
  such idealized errors exist by~(a).

  Statements (a) and (b) are the two clauses of \cref{def:family-tp} for the array $(\mathfrak{r}_{\mu_{ij}})$, with $k=2$, $\alpha=2/3$ and $\beta=1$.
\end{proof}

\begin{remark}\label{rem:qudits}
  For qudits of prime dimension $p$, the quantum systems with Hilbert space
  $\CC^p$, the labels form a symplectic space over the field
  $\FF_p$ with $p$ elements, and the response of a label at an error is $1$ when their commutation phase, the $p$th root of unity by which the products of the two strings in the two orders differ, is nontrivial and $0$ otherwise.
  The proof of \cref{thm:pauli-tp2} adapts as follows.
  A nonzero functional on a finite $\FF_p$-space takes each value equally often, so in~\eqref{eq:first-check} the conditions
  $\lambda(g_{ij})=1$ and $\lambda(g_{im})=1$ become $\lambda(g_{ij})\ne0$ and $\lambda(g_{im})\ne0$, and the response $\frac12$ becomes $\vartheta=\frac{p-1}p$.
  For $j<j'$, the pair of responses at a pure probe takes one of the values $(1,\vartheta)$, $(\vartheta,\vartheta)$, $(0,\vartheta)$, $(0,1)$ and $(0,0)$.
  The point $(\vartheta,\vartheta)$ lies below the segment from $(1,\vartheta)$ to $(0,1)$, since $\vartheta<1-\vartheta+\vartheta^2$, and on that segment the two coordinates are equal at the point with first coordinate $\frac p{p+1}$.
  Hence the bound $\frac23$ of \cref{thm:pauli-tp2}(b) becomes $\frac p{p+1}$, and it is attained.
\end{remark}

In the notation of \cref{thm:pauli-tp2} and the language of \cref{prop:sector}, the event “the check
$g_{ij}$ fires first in block $i$” is the syndrome sector of the checks $g_{i0},\dots,g_{ij}$ with syndrome $(0,\dots,0,1)$, whose two level sets are the coset $g_{ij}+W_{ij}$ and the group $W_{ij}$.
Hence its indicator is $\mathfrak{r}_{\mu_{ij}}-\mathfrak{r}_{\mu^0_{ij}}$ on pure probes, where $\mu^0_{ij}$ is the uniform probability measure on $W_{ij}$.
Only the linear independence of the checks enters \cref{thm:pauli-tp2};
when the checks commute pairwise, each $W_{ij}$ is the set of labels of a stabilizer group, and $\mu_{ij}$ draws uniformly from a coset of it.
The simplest instance is the repetition code.

\begin{example}\label{exa:repetition}
  Index the qubits by pairs $(i,m)\in\NN^2$ and call the qubits $(i,m)$ for $m\in\NN$ the chain $i$.
  Let $X_{i,m}$ and $Z_{i,m}$ denote the bit flip and phase flip on the qubit $(i,m)$.
  The check $g_{im}=Z_{i,m}Z_{i,m+1}$ compares the neighboring qubits $m$ and $m+1$ in the bit-flip repetition code on chain $i$.
  These checks commute pairwise and are linearly independent.
  Let $e\in\hat V$ be a bit-flip label with exponents $x_{i,m}\in\FF_2$ for $(i,m)\in\NN^2$.
  We write $e=\prod_{i,m}X_{i,m}^{x_{i,m}}$ as notation for this label, allowing infinitely many nonzero exponents;
  no infinite product of operators on a fixed Hilbert space is asserted.
  For the finite-support check $g_{im}$ we have $\omega(g_{im},e)=x_{i,m}+x_{i,m+1}$.
  This equals $1$ exactly when one of the neighboring qubits is flipped and the other is not, giving a domain wall between them.
  By~\eqref{eq:first-check}, the response $\mathfrak{r}_{\mu_{ij}}(e)$ of the uniform probability measure
  $\mu_{ij}$ on the coset of $g_{ij}$ modulo the span of $g_{i0},\dots,g_{i,j-1}$, as in \cref{thm:pauli-tp2}, is $1$ exactly when the first domain wall of $e$ on the chain $i$ lies between the qubits $j$ and $j+1$.
  For this code, \cref{thm:pauli-tp2} reads as follows:
  the chains are independent code blocks of a memory, an error cannot have two first domain walls on the same chain, and the positions of the first domain walls on finitely many chains can be prescribed independently by a single error.
  The bit flips on the chain $i$ with zero syndrome are the identity and the flip of every qubit of the chain, which acts on infinitely many qubits:
  the logical bit flip of an infinite chain is an idealized error, an element of $\hat{V}\setminus V$.
\end{example}

\begin{remark}\label{rem:fourier}
  The mechanism behind \cref{thm:pauli-tp2} is Fourier analysis on the compact group $V^*$:
  since the characters of $V^*$ separate points and are closed under multiplication, the linear span of Alice’s responses and the constants is dense in $C(V^*)$ by the Stone--Weierstrass theorem;
  it is an algebra, because $2\mathfrak{r}_a\mathfrak{r}_b=\mathfrak{r}_a+\mathfrak{r}_b-\mathfrak{r}_{a+b}$ for the pure responses $\mathfrak{r}_a=\mathfrak{r}_{\delta_a}$, $a\in V$, restricted to pure probes.
  By \cref{prop:sector}, the indicators of syndrome sectors, which are functions of supremum norm one, lie in this span with coefficients of total absolute value at most $3$.
  In the first-order structure $(V,\omega)$, by contrast, a formula speaks about a bounded number of labels, and, by \cref{thm:symp-simple}, no formula has the tree property.
\end{remark}

The coset mixtures of \cref{thm:pauli-tp2} thus strengthen the construction of \cref{thm:collapse}:
the response predicate itself witnesses the tree property of the second kind, with actual preparations as parameters, because a preparation of $\mathcal{M}_\omega$ can average over a finite set of labels whose size grows along the array, whereas a formula of the label structure $(V,\omega)$ speaks about a bounded number of labels.

\Cref{tab:response-dictionary} collects, for the Pauli observable, the translations of \cref{sec:observables,sec:collapse,sec:pauli-tp2}.
It concerns the real-valued structure $\mathcal{M}_\omega$ of preparations and probes of \cref{sec:observables}, for the predicate $\omega$, which has probability distributions on $V$ as preparations, and Radon probability measures on $\hat V=\prod_{k\in\NN}\FF_2^2$ as probes (see \cref{prop:channels}).
The compact space $\hat V$ consists of arbitrary-support error labels and contains $V$;
its point masses are the pure probes, with labels in $V$ giving actual finite-support errors.
The symplectic pairing of a finite-support check with any label in $\hat V$ is a finite sum, so every response is defined (see \cref{prop:channels}).
A specified nonzero syndrome is a linear functional on the finite check span;
its two level sets, the kernel and the complementary coset, are labeled by the values $0$ and $1$.

\begin{table}[ht]
  \centering
  \small
  \renewcommand{\arraystretch}{1.22}
  \setlength{\tabcolsep}{3pt}
  \arrayrulecolor{urlpurple!55!white}
  \begin{tabular}{@{}>{\raggedright\arraybackslash}p{0.27\textwidth}>{\raggedright\arraybackslash}p{0.51\textwidth}>{\raggedright\arraybackslash}p{0.18\textwidth}@{}}
    \toprule
    \rowcolor{palepurple}
    \textbf{Response structure} & \textbf{Pauli interpretation} & \textbf{Result} \\
    \midrule
    actual error labels $e\in V$ & point masses at finite-support errors & \cref{prop:channels} \\
    all pure probes & point masses at arbitrary-support labels in $\hat V$ & \cref{prop:channels} \\
    all probes & Pauli channels: probability laws on $\hat V$ & \cref{prop:channels} \\
    response of a mixture & anticommutation probability for independent draws of a check and an error & \cref{prop:channels} \\
    uniform response on the $1$-coset minus that on the $0$-coset & probability of the specified nonzero syndrome & \cref{prop:sector} \\
    commuting signed checks & operator analogue: projector onto a syndrome eigenspace & \cref{prop:sector} \\
    TP$_2$ parameters & actual finite mixtures of checks & \cref{thm:pauli-tp2} \\
    path realizations & finite path: an error in $V$; whole path: a pure probe in $\hat V$ & \cref{thm:pauli-tp2} \\
    \bottomrule
  \end{tabular}
  \vspace{6pt}
  \caption{From response functions to syndrome probabilities: the response structure and its Pauli interpretation.}
  \label{tab:response-dictionary}
\end{table}

\section{The Rado observable: marginals and idealized preparations}\label{sec:free}

For the Rado observable, the mixed theory has the same tree property as for Pauli, but the preparation parameters required to witness it are different.
The pure probe space is the full cube of adjacency patterns, and the response of an actual mixture depends only on the probabilities with which a probe declares individual vertices to be neighbors.
The freedom to prescribe these marginals separately on disjoint blocks prevents an infinite uniformly inconsistent row with actual mixture parameters,
first for bounded affine combinations and then for each fixed continuous connective applied to finitely many responses (see \cref{thm:free,thm:free-qf}).
Consequently, a witnessing array for a quantifier-free formula of this form, in a probe variable with preparation parameters, must contain preparation parameters from an elementary extension outside the original structure.
The restriction concerns the parameters of the array;
the probes in the compatibility results range over the whole probe space, including mixed probes and pure probes that are not actual labels.
The same argument applies to the random bipartite graph, whose two vertex classes serve as Alice’s and Bob’s labels.

\subsection{Affine combinations of responses}
Consider a diagonal observable with freely independent pure responses, that is, one for which every pattern $\lambda\in\{0,1\}^A$ of answers to Alice’s labels, the elements of $A$, is an idealized label of Bob.
The response of a mixture depends on a probe $\nu$ only through its marginals
$\pi_a(\nu)=\nu\{\lambda:\lambda(a)=1\}$, for $a\in A$, as for the Rado observable in \cref{sec:rado-obs}, and every vector of marginals in $[0,1]^A$ is attained by a product measure.
A bounded affine combination of responses is therefore a function
\[
  h_{c,\gamma}(\nu)=\gamma+\sum_{a\in A}c(a)\,\pi_a(\nu),\qquad \gamma\in\RR,\ c\in\ell^1(A),
\]
and its supremum over the probe space is $\gamma+\sum_a c(a)^+$, attained at the pure probe $\lambda=\one_{\{c>0\}}$.
Two such functions can conflict only through the labels on which both place weight.
A gliding hump, in the proof of \cref{thm:free} below, shows that a sequence of such functions along which
$|\gamma|+\|c\|_1$ stays bounded has a subsequence whose members place their weight, up to small errors, on a common finite set of labels and on pairwise disjoint finite blocks;
for every positive integer $k$, a single pure probe then nearly maximizes $k$ of them.

\begin{keytheorem}[Compatibility of affine conditions]\label{thm:free}
  Let $C_\varphi$ be the diagonal observable of a function $\varphi\colon A\times B\to\{0,1\}$ on countable sets $A$ and $B$, with freely independent pure responses, and let $\kappa>0$ be given.
  Let $\cR_\kappa$ be the family of functions $\gamma+\sum_{i<n}c_i\mathfrak{r}_{\mu_i}$
  on the probe space $\cP(\{0,1\}^A)$ of Radon probability measures on $\{0,1\}^A$, where $n\in\NN$, the $\mu_i$ are mixtures of Alice, and $\gamma$ and the $c_i$ are reals with $|\gamma|+\sum_i|c_i|\le\kappa$.
  Then the family $\cR_\kappa$ contains no uniformly $k$-inconsistent sequence, for any integer $k\ge2$.
  In particular, no tree property is witnessed in $\cR_\kappa$.
\end{keytheorem}

\begin{proof}
  Let $(h_j)_{j\in\NN}$ be a sequence in $\cR_\kappa$ and let $\alpha<\beta$ be such that each $h_j$ takes a value $\ge\beta$.
  We show that for every integer $k\ge2$ there are $j_1<\dots<j_k$ and a pure probe $\lambda$ with $h_{j_\ell}(\delta_\lambda)>\alpha$ for $\ell=1,\dots,k$.

  Write $h_j=\gamma_j+\sum_ic_{j,i}\mathfrak{r}_{\mu_{j,i}}$, with real coefficients $c_{j,i}$ and mixtures $\mu_{j,i}$ of Alice;
  then $h_j=h_{c_j,\gamma_j}$ with $c_j=\sum_ic_{j,i}\mu_{j,i}$, whence $\|c_j\|_1\le\kappa$.
  Since $A$ is countable, a diagonal argument gives a subsequence, still denoted $(c_j)$, that converges pointwise on $A$ to some $c$;
  by Fatou’s lemma, the limit satisfies $\|c\|_1\le\kappa$.
  Let $\eta=(\beta-\alpha)/5$, and choose a finite set $F\subseteq A$ with $\|c\restriction(A\setminus F)\|_1<\eta$.
  Let $F_0=F$, and choose inductively indices $j_1<\dots<j_k$ and finite sets $F_0\subseteq F_1\subseteq\dots\subseteq F_k$ such that
  \[
    \|(c_{j_\ell}-c)\restriction F_{\ell-1}\|_1<\eta
    \quad\text{and}\quad
    \|c_{j_\ell}\restriction(A\setminus F_\ell)\|_1<\eta;
  \]
  the first condition holds for all large indices by pointwise convergence, and the second for a large enough finite set.
  Define $\lambda\in\{0,1\}^A$ by $\lambda(a)=1$ exactly when $a\in F$ and $c(a)>0$, or $a\in F_\ell\setminus F_{\ell-1}$ and $c_{j_\ell}(a)>0$ for some $\ell\le k$.

  Let $\ell\le k$ be given.
  The deficit of $\lambda$ for $h_{j_\ell}$ is
  \[
    \sup h_{j_\ell}-h_{j_\ell}(\delta_\lambda)=
    \sum_{\lambda(a)=0}c_{j_\ell}(a)^++\sum_{\lambda(a)=1}c_{j_\ell}(a)^-,
  \]
  and we bound its contribution on four disjoint pieces of $A$.
  On $F_\ell\setminus F_{\ell-1}$ the sign pattern of $\lambda$ is that of $c_{j_\ell}$, and the contribution is $0$.
  On $F$ it is that of $c$, and a label where the signs of $c_{j_\ell}$ and $c$ disagree contributes at most $|c_{j_\ell}(a)-c(a)|$;
  the total is less than $\eta$.
  On $F_{\ell-1}\setminus F$ the contribution is at most $\|c_{j_\ell}\restriction(F_{\ell-1}\setminus F)\|_1 \le\|(c_{j_\ell}-c)\restriction F_{\ell-1}\|_1+\|c\restriction(A\setminus F)\|_1<2\eta$.
  On $A\setminus F_\ell$ it is less than $\eta$.
  Hence $h_{j_\ell}(\delta_\lambda)>\sup h_{j_\ell}-4\eta\ge\beta-4\eta>\alpha$.
\end{proof}

\begin{corollary}\label{cor:rado}
  Let $n\in\NN$ be given, and let $\gamma$ and $c_i$, for $i<n$, be reals.
  For the Rado observable, no tree property of the affine formula
  $\gamma+\sum_{i<n}c_iR(y_i,x)$, in which $R$ is the response predicate, the variables
  $y_i$ are of the sort of preparations and the variable
  $x$ is of the sort of probes, is witnessed by actual preparations, although the real-valued theory of the structure $\mathcal{M}_E$ of its mixed preparations and probes has the tree property of the second kind.
\end{corollary}

\begin{proof}
  By \cref{lem:rado-probes}, the Rado observable has freely independent pure responses.
  With actual preparations as parameters, the instances of the formula lie in the family $\cR_\kappa$ of \cref{thm:free} for every $\kappa>0$ with $\kappa\ge|\gamma|+\sum_{i<n}|c_i|$, and the theorem yields the first statement.
  The second statement is \cref{thm:collapse}.
\end{proof}

The distinction between finite and infinite arrays is essential here.
For the fixed difference-of-responses formula of \cref{thm:collapse}, actual mixtures provide arrays of every finite size, but \cref{thm:free} excludes an infinite uniformly inconsistent row whose preparation parameters are all actual mixtures.
Hence each infinite row supplied by compactness must contain an idealized preparation parameter;
the arbitrarily large finite arrays do not assemble into such a row inside the original structure.
This conclusion also holds for the random bipartite graph, since its finite extension property makes every adjacency pattern a pure probe.

\subsection{Quantifier-free formulas}
In \cref{thm:free} the conditions on a probe are given by affine combinations of responses.
In the real-valued logic of Part~II, a quantifier-free formula in a variable
$x$ of the sort of probes, built from the response with parameters of the sort of preparations, is a continuous function of finitely many responses, since the responses take values in $[0,1]$ and the connectives are continuous functions.
The next theorem is the analogue of \cref{thm:free} for these formulas.
Its proof uses the same gliding hump, with one change:
a continuous function of responses is no longer maximized by reading off signs (we show this after the proof), so instead of building the probe from signs we glue maximizers of the individual instances on disjoint blocks of labels.

\begin{keytheorem}[Compatibility of quantifier-free conditions]\label{thm:free-qf}
  Let $C_\varphi$ be the diagonal observable of a function $\varphi\colon A\times B\to\{0,1\}$ on countable sets $A$ and $B$, with freely independent pure responses, let $n\ge1$ be an integer, and let $u\colon[0,1]^n\to\RR$ be a continuous function.
  Let $\cR_u$ be the family of functions
  \[
    h(\nu)=u\bigl(\mathfrak{r}_{\mu_1}(\nu),\dots,\mathfrak{r}_{\mu_n}(\nu)\bigr)
    \qquad(\nu\in\cP(\{0,1\}^A)),
  \]
  where $\mu_1,\dots,\mu_n$ range over the mixtures of Alice.
  Then the family $\cR_u$ contains no uniformly $k$-inconsistent sequence, for any integer $k\ge2$.
  In particular, no tree property of the quantifier-free formula
  $u(R(y_1,x),\dots,R(y_n,x))$, where $R$ is the response predicate, the variable $x$ is of the sort of probes and the variables $y_1,\dots,y_n$ are of the sort of preparations, is witnessed by actual preparations.
\end{keytheorem}

\begin{proof}
  For $\nu\in\cP(\{0,1\}^A)$ write $\pi(\nu)\in[0,1]^A$ for its vector of marginals $(\pi_a(\nu))_{a\in A}$, defined before \cref{thm:free}.
  Every $h\in\cR_u$ factors as $h(\nu)=H(\pi(\nu))$ with
  \[
    H(q)=u\bigl(\langle\mu_1,q\rangle,\dots,\langle\mu_n,q\rangle\bigr),
    \qquad \langle\mu,q\rangle=\sum_{a\in A}\mu(a)q_a ,
  \]
  and every $q\in[0,1]^A$ is the vector of marginals of a product measure.
  It therefore suffices to work on the cube $[0,1]^A$.

  Let $(h_j)_{j\in\NN}$ be a sequence in $\cR_u$, with parameters $\mu_{j,1},\dots,\mu_{j,n}$ and associated functions $H_j$ on the cube, and let $\alpha<\beta$ be such that each $h_j$ takes a value $\ge\beta$.
  We show that for every integer $k\ge2$ there are $j_1<\dots<j_k$ and $q\in[0,1]^A$ with $H_{j_\ell}(q)>\alpha$ for $\ell=1,\dots,k$.
  Since $u$ is uniformly continuous, there is $\theta>0$ such that $|u(s)-u(t)|<\beta-\alpha$ whenever $\max_i|s_i-t_i|<\theta$.
  For each $j$ choose $q^{(j)}\in[0,1]^A$ with $H_j(q^{(j)})\ge\beta$.
  The set $A$ is countable and the cube is compact and metrizable, so a diagonal argument provides a subsequence, still indexed by $j$, along which $\mu_{j,i}\to\mu_i$ pointwise on $A$ for each $i\le n$ and $q^{(j)}\to q^\infty$ in $[0,1]^A$.
  By Fatou’s lemma, each limit $\mu_i$ has total mass at most $1$.
  Let $\eta=\theta/5$, and choose a finite set $F\subseteq A$ with $\mu_i(A\setminus F)<\eta$ for all $i\le n$.

  Let $F_0=F$, and choose inductively indices $j_1<\dots<j_k$ and finite sets $F_0\subseteq F_1\subseteq\dots\subseteq F_k$ such that, for all $i\le n$,
  \[
    \sum_{a\in F_{\ell-1}}|\mu_{j_\ell,i}(a)-\mu_i(a)|<\eta,\qquad
    \sum_{a\in F}|q^{(j_\ell)}_a-q^\infty_a|<\eta,\qquad
    \mu_{j_\ell,i}(A\setminus F_\ell)<\eta .
  \]
  The first two conditions hold for all large indices, by convergence on the finite sets $F_{\ell-1}$ and $F$, and the third for a large enough finite set $F_\ell$.
  Define $q\in[0,1]^A$ by $q_a=q^\infty_a$ for $a\in F$, $q_a=q^{(j_\ell)}_a$ for $a\in F_\ell\setminus F_{\ell-1}$ and $\ell\le k$, and $q_a=0$ for $a\notin F_k$.

  Let $\ell\le k$ and $i\le n$ be given.
  We compare $\langle\mu_{j_\ell,i},q\rangle$ with $\langle\mu_{j_\ell,i},q^{(j_\ell)}\rangle$, on the same four pieces of $A$ as in the proof of \cref{thm:free};
  on each piece the two vectors differ by at most $1$ in each coordinate.
  On $F_\ell\setminus F_{\ell-1}$ they agree.
  On $F$ the difference is at most $\sum_{a\in F}|q^\infty_a-q^{(j_\ell)}_a|<\eta$, since $\mu_{j_\ell,i}(a)\le1$.
  On $F_{\ell-1}\setminus F$ it is at most $\mu_{j_\ell,i}(F_{\ell-1}\setminus F)\le\sum_{a\in F_{\ell-1}}|\mu_{j_\ell,i}(a) -\mu_i(a)|+\mu_i(A\setminus F)<2\eta$.
  On $A\setminus F_\ell$ it is at most $\mu_{j_\ell,i}(A\setminus F_\ell)<\eta$.
  Hence $|\langle\mu_{j_\ell,i},q\rangle-\langle\mu_{j_\ell,i},q^{(j_\ell)}\rangle|<4\eta<\theta$ for every $i\le n$, and the choice of $\theta$ gives $H_{j_\ell}(q)>H_{j_\ell}(q^{(j_\ell)})-(\beta-\alpha)\ge\alpha$.
\end{proof}

For affine combinations, the proof of \cref{thm:free} yields slightly more:
there the common witness can be taken pure.
For an arbitrary connective $u$ the witness is the product probe whose vector of marginals is the point $q$ constructed in the proof of \cref{thm:free-qf}, and it cannot always be taken pure:
for $n=1$, $\mu=\delta_a$ with $a\in A$, and $u(t)=1-|2t-1|$, the instance vanishes at every pure probe and equals $1$ at the probes $\nu$ with $\pi_a(\nu)=1/2$.

Thus, for the Rado observable, the freedom to prescribe marginals on disjoint blocks of labels lets the gliding hump assemble compatible conditions from every sequence with actual parameters.
For the Pauli observable, by contrast, the linear relations among label coordinates place the indicators of syndrome sectors in the affine span of the responses, with a coefficient bound independent of the number of checks (see \cref{prop:sector}), and these identities yield the actual-parameter array of \cref{thm:pauli-tp2}.

\section{Affine logic and correlated preparations}\label{sec:affine}

Finally, we place the results in affine logic,
in the form developed by Ben Yaacov, Ibarlucía and Tsankov \cite{BenYaacov-Ibarlucia-Tsankov:2024}, whose connectives are the affine maps and in which a mixture of structures is a direct integral.
Formula values in a direct integral are integrals of formula values, which is the rule by which a response is computed on a mixed preparation.
In this logic the type spaces are compact convex sets, and a structure is \emph{extremal} when it realizes only extreme types, the extreme points of these sets.
Ben Yaacov, Ibarlucía and Tsankov \cite{BenYaacov-Ibarlucia-Tsankov:2024} showed that every model of an affine theory is an affine elementary substructure of a direct integral of extremal models.
An affine theory is simplicial when its type spaces are Choquet simplices, that is, sets in which every point is the barycenter of exactly one probability measure that is maximal in the Choquet order;
when the set is metrizable, the maximal measures are those concentrated on its extreme points.
The same authors also showed that every model of a simplicial theory is such a direct integral, uniquely.

Ben Yaacov and Ibarlucía \cite{BenYaacov-Ibarlucia:2026} proved that an affine formula stable in almost every member of a measurable field of structures is stable in its direct integral,
and that an affine theory is stable exactly when its extremal models are stable.
Their definition of stability of a formula in a structure is Grothendieck’s double limit exchange condition, which is also the definition used in Part~I~\cite{Duenez-Iovino:2026a}:
for every two sequences of elements, one in each variable of the formula, the two iterated limits of its values agree whenever both exist.
Mixing in the sense of this paper behaves differently with respect to simplicity:
the first-order theories of the Pauli group and of the Rado graph are simple, whereas, by \cref{thm:collapse}, the real-valued theories of their observables have the tree property of the second kind.

The distinction between the two forms of averaging is already visible in the expectation kernel.
In a direct integral, the elements in both variables may depend on the same random parameter, and the formula value is averaged over that common parameter;
for a label relation, this permits correlated choices of Alice’s and Bob’s labels.
Keisler’s randomizations \cite{Keisler:1999,BenYaacov:2013} likewise use random elements defined on a common probability space.
In the observable structure considered here, Alice’s label and Bob’s pure probe are drawn independently from their respective laws, and the response is their product average (see \cref{lem:probe-space}).
The tree property established in \cref{thm:collapse} therefore already occurs with independent product preparations.

\section{Questions}\label{sec:questions}

\begin{question}\label{ques:criterion}
  Which diagonal observables with the independence property have the tree property of the second kind witnessed by actual preparations?
  For the diagonal observable $C_\varphi$ of a function
  $\varphi\colon A\times B\to\{0,1\}$ on countable sets
  $A$ and $B$, \cref{prop:sector}, \cref{thm:pauli-tp2} and \cref{thm:free} suggest a criterion in terms of the linear span of its responses in the space $C(K_\varphi)$ of continuous functions on the space $K_\varphi$ of idealized labels of Bob:
  whether it contains, with coefficients of bounded total absolute value,
  the indicators of the sets of an infinite array of clopen sets whose rows consist of pairwise disjoint sets and in which every choice of one set from each of finitely many rows has nonempty intersection.
\end{question}

\begin{question}\label{ques:qudits}
  By \cref{rem:qudits}, the proof of \cref{thm:pauli-tp2} adapts to qudits of prime dimension.
  For continuous variables, that is, for systems such as a particle on a line, whose position and momentum take real values, the commutation phase of a translation by a distance
  $x$ and a momentum kick by $\hbar k$, where $\hbar$ is the reduced Planck constant and the wavenumber $k$ is a real number measured in the inverse of the unit of length, is $e^{\mathrm{i}kx}$, whose real part $\cos kx$ describes an interference fringe.
  By Part~II~\cite{Duenez-Iovino:2026b}, the fringe observable, which asks whether a photon lands on a bright fringe of an interference pattern of wavenumber $8^n$, for a label $n\in\NN$ of Alice, is not NIP.
  Where do interference fringes of unbounded wavenumber lie with respect to simplicity?
\end{question}

\begin{question}\label{ques:nip-simple}
  In first-order model theory, a simple theory without the independence property is stable.
  Is an observable stable when its real-valued theory is simple?
  By \cref{thm:collapse}, every simple diagonal observable is NIP,
  so for diagonal observables the question is whether those of Part~II that are NIP but unstable,
  such as the comparison of photon numbers, which asks whether Bob’s system has more photons than Alice’s, have a tree property once preparations are mixed.
\end{question}

\begin{question}\label{ques:general}
  Does \cref{thm:collapse} hold for every observable whose responses have the independence property, not only for diagonal ones?
  The proof uses the values $0$ and $1$ of pure responses;
  for general observables they are replaced by two distinct levels, and the averages must be controlled with a gap.
\end{question}

\begin{question}\label{ques:affine-kp}
  Is there an analogue of the Kim--Pillay theorem for affine logic, characterizing simple affine theories by an independence relation on extremal types?
\end{question}

\section*{Historical remarks}

Simple theories were introduced by Shelah \cite{Shelah:1980b} in 1980 with the theory of the countable random graph as a motivating example (see also \cite{Grossberg-Iovino-Lessmann:2002}).
The tree properties of the first and second kind and their dichotomy are treated in Section~III.7 of his monograph \cite{Shelah:1990a}.
Kim and Pillay \cite{Kim-Pillay:1997} proved the
independence theorem and the characterization of simplicity by an independence relation in 1997, and Kim \cite{Kim:1998} proved in 1998 that forking is symmetric in simple theories.
Ben Yaacov \cite{BenYaacov:2003b} extended the theory in 2003 to compact abstract theories, a framework that contains the continuous logic of metric structures.
The Rado graph was described by Ackermann \cite{Ackermann:1937} in 1937 and by Rado \cite{Rado:1964} in 1964;
by a theorem of Erdős and Rényi \cite{Erdos-Renyi:1963} from 1963, it is the almost sure isomorphism type of the random graph on countably many vertices.
Fraïssé \cite{Fraisse:1954a} introduced his limits in 1954, and Witt’s theorem on isometries of forms is in Artin’s \emph{Geometric Algebra} \cite{Artin:1957}.
Vector spaces of countably infinite dimension over finite fields with a nondegenerate alternating form, such as the space of finite-support Pauli labels, are among the Cherlin--Hrushovski smoothly approximable structures \cite{Cherlin-Hrushovski:2003}.
Granger \cite{Granger:1999} studied vector spaces with a generic bilinear form over an arbitrary field in his thesis of 1999 and showed that, over an infinite field carried as a second sort, the theory is not simple;
Chernikov and Ramsey \cite{Chernikov-Ramsey:2016} proved in 2016 that over an algebraically closed field it is NSOP$_1$:
it lacks the $1$-strong order property, as every simple theory does.
The symplectic description of the Pauli group and of stabilizer codes is due to Gottesman \cite{Gottesman:1997} and to Calderbank, Rains, Shor and Sloane \cite{Calderbank-Rains-Shor-Sloane:1998};
the expression of the projector onto the states of a given syndrome as an average over the stabilizer group, behind \cref{prop:sector}, is standard in Gottesman’s stabilizer formalism.
Keisler \cite{Keisler:1999} introduced randomizations in 1999;
Ben Yaacov and Keisler \cite{BenYaacov-Keisler:2009a} treated them as metric structures, and Ben Yaacov \cite{BenYaacov:2013} proved in 2013 that the Keisler randomization of a theory with the independence property has TP$_2$.
Direct integrals and extremal models in affine logic are due to Ben Yaacov, Ibarlucía and Tsankov \cite{BenYaacov-Ibarlucia-Tsankov:2024}, and stability in affine logic to Ben Yaacov and Ibarlucía \cite{BenYaacov-Ibarlucia:2026}.
The first two papers of this series are \cite{Duenez-Iovino:2026a,Duenez-Iovino:2026b}.

\section*{Declarations}
\noindent\emph{Competing interests.}
The authors have no competing interests to
declare.

\noindent\emph{Data availability.}
No data were generated or analyzed in this
work.

\bibliographystyle{amsalpha}
\bibliography{references}

\def\cprime{$'$}
\providecommand{\bysame}{\leavevmode\hbox to3em{\hrulefill}\thinspace}
\providecommand{\MR}{\relax\ifhmode\unskip\space\fi MR }
\providecommand{\MRhref}[2]{%
  \href{http://www.ams.org/mathscinet-getitem?mr=#1}{#2}
}
\providecommand{\href}[2]{#2}
\begin{thebibliography}{BDSW96}

\bibitem[Ack37]{Ackermann:1937}
W.~Ackermann, \emph{Die {Widerspruchsfreiheit} der allgemeinen {Mengenlehre}},
  Math. Ann. \textbf{114} (1937), 305--315.

\bibitem[Art57]{Artin:1957}
E.~Artin, \emph{Geometric {Algebra}}, Interscience, New York, 1957.

\bibitem[BDSW96]{Bennett-DiVincenzo-Smolin-Wootters:1996}
C.~H. Bennett, D.~P. DiVincenzo, J.~A. Smolin, and W.~K. Wootters,
  \emph{Mixed-state entanglement and quantum error correction}, Phys. Rev. A
  \textbf{54} (1996), 3824--3851.

\bibitem[BY03]{BenYaacov:2003b}
I.~Ben~Yaacov, \emph{Simplicity in compact abstract theories}, J. Math. Log.
  \textbf{3} (2003), 163--191.

\bibitem[BY13]{BenYaacov:2013}
Ita\"{\i} Ben~Yaacov, \emph{On theories of random variables}, Isr. J. Math.
  \textbf{194} (2013), 957--1012.

\bibitem[BYI26]{BenYaacov-Ibarlucia:2026}
I.~Ben~Yaacov and T.~Ibarluc{\'i}a, \emph{Stability in affine logic}, Model
  Theory \textbf{5} (2026), 1--34.

\bibitem[BYIT24]{BenYaacov-Ibarlucia-Tsankov:2024}
I.~Ben~Yaacov, T.~Ibarluc{\'i}a, and T.~Tsankov, \emph{Extremal models and
  direct integrals in affine logic}, arXiv:2407.13344, 2024.

\bibitem[BYK09]{BenYaacov-Keisler:2009a}
Ita\"{\i} Ben~Yaacov and H.~Jerome Keisler, \emph{Randomizations of models as
  metric structures}, Confluentes Math. \textbf{1} (2009), no.~2, 197--223.
  \MR{2561997}

\bibitem[CH03]{Cherlin-Hrushovski:2003}
G.~Cherlin and E.~Hrushovski, \emph{Finite {Structures} with {Few} {Types}},
  Annals of Mathematics Studies, vol. 152, Princeton University Press, 2003.

\bibitem[CR16]{Chernikov-Ramsey:2016}
A.~Chernikov and N.~Ramsey, \emph{On model-theoretic tree properties}, J. Math.
  Log. \textbf{16} (2016), 1650009.

\bibitem[CRSS98]{Calderbank-Rains-Shor-Sloane:1998}
A.~R. Calderbank, E.~M. Rains, P.~W. Shor, and N.~J.~A. Sloane, \emph{Quantum
  error correction via codes over {GF(4})}, IEEE Trans. Inform. Theory
  \textbf{44} (1998), 1369--1387.

\bibitem[DDM03]{Dehaene-deMoor:2003}
J.~Dehaene and B.~De~Moor, \emph{Clifford group, stabilizer states, and linear
  and quadratic operations over {GF(2})}, Phys. Rev. A \textbf{68} (2003),
  042318.

\bibitem[DI26a]{Duenez-Iovino:2026a}
E.~Due{\~n}ez and J.~Iovino, \emph{Classification of bipartite observables {I}:
  stable and uniformly stable observables}, preprint; arXiv:2610.11002, 2026.

\bibitem[DI26b]{Duenez-Iovino:2026b}
\bysame, \emph{Classification of bipartite observables {II}: {NIP} observables
  and {Todor\v{c}evi\'c}'s trichotomy}, preprint; arXiv:2610.12908, 2026.

\bibitem[ER63]{Erdos-Renyi:1963}
P.~Erd{\H{o}}s and A.~R{\'e}nyi, \emph{Asymmetric graphs}, Acta Math. Acad.
  Sci. Hungar. \textbf{14} (1963), 295--315.

\bibitem[Fra54]{Fraisse:1954a}
R.~Fra{\"i}ss{\'e}, \emph{Sur l'extension aux relations de quelques
  propri\'et\'es des ordres}, Ann. Sci. \'Ecole Norm. Sup. \textbf{71} (1954),
  363--388.

\bibitem[GIL02]{Grossberg-Iovino-Lessmann:2002}
Rami Grossberg, Jos{\'e}~N. Iovino, and Olivier Lessmann, \emph{A primer of
  simple theories}, Arch. Math. Logic \textbf{41} (2002), no.~6, 541--580.
  \MR{MR1923196 (2003f:03043)}

\bibitem[Got97]{Gottesman:1997}
D.~Gottesman, \emph{Stabilizer codes and quantum error correction}, Ph.D.
  thesis, California Institute of Technology, 1997, arXiv:quant-ph/9705052.

\bibitem[Gra99]{Granger:1999}
N.~Granger, \emph{Stability, simplicity and the model theory of bilinear
  forms}, Ph.D. thesis, University of Manchester, 1999.

\bibitem[Kei99]{Keisler:1999}
H.~Jerome Keisler, \emph{Randomizing a model}, Adv. Math. \textbf{143} (1999),
  no.~1, 124--158. \MR{1680650}

\bibitem[Kei23]{Keisler:2023}
H.~Jerome Keisler, \emph{Model theory for real-valued structures}, Beyond
  First-Order Model Theory, Vol.\ 2 (Jos{\'e} Iovino, ed.), CRC Press, 2023,
  arXiv:2005.11851, pp.~43--101.

\bibitem[Kim98]{Kim:1998}
B.~Kim, \emph{Forking in simple unstable theories}, J. London Math. Soc. (2)
  \textbf{57} (1998), 257--267.

\bibitem[KP97]{Kim-Pillay:1997}
Byunghan Kim and Anand Pillay, \emph{Simple theories}, Ann. Pure Appl. Logic
  \textbf{88} (1997), no.~2-3, 149--164, Joint AILA-KGS Model Theory Meeting
  (Florence, 1995). \MR{MR1600895 (99b:03049)}

\bibitem[LMPZ96]{Laflamme-Miquel-Paz-Zurek:1996}
R.~Laflamme, C.~Miquel, J.~P. Paz, and W.~H. Zurek, \emph{Perfect quantum error
  correcting code}, Phys. Rev. Lett. \textbf{77} (1996), 198--201.

\bibitem[Rad64]{Rado:1964}
R.~Rado, \emph{Universal graphs and universal functions}, Acta Arith.
  \textbf{9} (1964), 331--340.

\bibitem[Rod21]{Rodriguez:2021}
Jos{\'e} Rodr{\'i}guez, \emph{$\varepsilon$-weakly precompact sets in {B}anach
  spaces}, 2021, arXiv:2102.01546.

\bibitem[Ros74]{Rosenthal:1974}
H.~P. Rosenthal, \emph{A characterization of {B}anach spaces containing
  $l\sp{1}$}, Proc. Nat. Acad. Sci. U.S.A. \textbf{71} (1974), 2411--2413.
  \MR{358307}

\bibitem[She80]{Shelah:1980b}
S.~Shelah, \emph{Simple unstable theories}, Annals of Math. Logic \textbf{19}
  (1980), 177--203.

\bibitem[She90]{Shelah:1990a}
S.~Shelah, \emph{Classification theory and the number of nonisomorphic models},
  second ed., Studies in Logic and the Foundations of Mathematics, vol.~92,
  North-Holland Publishing Co., Amsterdam, 1990. \MR{MR1083551 (91k:03085)}

\end{thebibliography}

\end{document}